\documentclass[12pt]{amsart}
\usepackage[utf8]{inputenc}
\usepackage[english]{babel}

\usepackage{amssymb,amsmath,amsthm,amsfonts,xcolor,enumerate,hyperref,comment,longtable,cleveref}
\usepackage{soul}
\usepackage{times}
\usepackage{cite}
\usepackage{pdflscape}
\usepackage{ulem}
\usepackage[mathcal]{euscript}
\usepackage{tikz}
\usepackage{hyperref}
\usepackage{cancel}
\usepackage{stmaryrd}

\usetikzlibrary{arrows}

\newcommand{\Q}{{\mathbb A}}

\newcommand{\A}{\mathcal{A}}  
 
\newtheorem{theorem}{Theorem}
\newtheorem{lemma}[theorem]{Lemma}

\newtheorem{remark}[theorem]{Remark}

\newtheorem{definition}[theorem]{Definition}

\newtheorem*{theoremA}{Theorem A}

\newtheorem*{theoremB}{Theorem B}

\newenvironment{Proof}[1][Proof.]{\begin{trivlist}
\item[\hskip \labelsep {\bfseries #1}]}{\flushright
$\Box$\end{trivlist}}

\usepackage{stmaryrd}
\usepackage{xcolor}

\newcommand{\la}{\langle}
\newcommand{\ra}{\rangle}

\newcommand{\0}{\theta}

\begin{document}

\noindent{\Large
The algebraic and geometric classification of 
antiassociative algebras}
   \medskip

   {\bf  
   Renato Fehlberg J\'{u}nior,$^{a}$
    Ivan Kaygorodov$^{b}$   $\&$
    Crislaine Kuster$^{c}$\\

    \medskip
}

{\tiny

$^{a}$  Departamento de Matemática, Universidade Federal do Espírito Santo, Vitória - ES, Brazil

$^{b}$ Centro de Matemática e Aplicações, Universidade da Beira Interior, Covilhã, Portugal

$^{c}$ Instituto de Matemática Pura e Aplicada, Rio de Janeiro - RJ, Brazil

\medskip

   E-mail addresses:

\smallskip Renato Fehlberg J\'{u}nior (renato.fehlberg@ufes.br)

\smallskip    Ivan Kaygorodov (kaygorodov.ivan@gmail.com)

Crislaine Kuster (crislainekeizy@gmail.com)
}

\bigskip 

\medskip 

\medskip

\noindent{\bf Abstract}:
{\it This paper is devoted to the complete algebraic and geometric classification of complex $4$ and $5$-dimensional antiassociative algebras.
In particular, 
we proved that 
the variety of complex $4$-dimensional antiassociative algebras has dimension $12$ and it is defined by three irreducible components
(in particular, there is only $1$ rigid algebra in this variety);
the variety of complex $5$-dimensional antiassociative algebras has dimension $24$ and it is defined by $8$ irreducible components
(in particular, there are only $4$ rigid algebras in this variety). 
 }
 
\medskip


\medskip


 \medskip

 \medskip

\section*{Introduction}
The algebraic classification (up to isomorphism) of algebras of dimension $n$ from a certain variety
defined by a certain family of polynomial identities is a classic problem in the theory of non-associative algebras.
There are many results related to the algebraic classification of small-dimensional algebras in the varieties of
Jordan, Lie, Leibniz, Zinbiel and many other algebras \cite{    kkl20,    hac16,         degr2,        kkp20}.
 Geometric properties of a variety of algebras defined by a family of polynomial identities have been an object of study since 1970's (see, \cite{wolf2,      chouhy,     BC99,    aleis2,   gabriel,   ckls, cibils,  shaf, GRH, GRH2, ale3,        kppv}). 
 Gabriel described the irreducible components of the variety of $4$-dimensional unital associative algebras~\cite{gabriel}.  
 Cibils considered rigid associative algebras with $2$-step nilpotent radical \cite{cibils}.
 Burde and Steinhoff  constructed the graphs of degenerations for the varieties of    $3$-dimensional and $4$-dimensional Lie algebras~\cite{BC99}. 
 Grunewald and O'Halloran  calculated the degenerations for the variety of $5$-dimensional nilpotent Lie algebras~\cite{GRH}. 
Chouhy  proved that  in the case of finite-dimensional associative algebras,
 the $N$-Koszul property is preserved under the degeneration relation~\cite{chouhy}.
Degenerations have also been used to study a level of complexity of an algebra~\cite{  wolf2}.
 The study of degenerations of algebras is very rich and closely related to deformation theory, in the sense of Gerstenhaber \cite{ger63}.
 
In the present paper, we give the algebraic  and geometric classification of
complex $4$-dimensional and $5$-dimensional antiassociative algebras.
An antiassociative algebra is  defined by the following identity $(xy)z+x(yz)=0.$
Thanks to \cite{xyzw=0}, each antiassociative algebra is a $3$-nilpotent algebra and we can use the standard method of classification of nilpotent algebras first appeared in \cite{ss78}.
On the other hand, the variety of antiassociative algebras  is a particular case of $(\alpha,\beta,\gamma)$-algebras considered  by Bergen and Grzeszczuk in \cite{BG1,BG2}, and by Calderón in \cite{calderon}.
It is also a particular case of Lie type algebras considered by Makarenko in \cite{natalia}, 
which is a particular case of $(\alpha,\beta,\gamma)$-algebras.
It looks that antiassociative algebras were initially introduced in  \cite{Shestakov}. 
Later, they appeared in several articles, both with a study of theirs properties, see \cite{xyzw=0,HoHA20}, as well as with theirs relations with others algebras,  like Jordan--Lie superalgebras or  CB-algebras  \cite{xyzw=0,cb-algebras}. In particular, an  anticommutative algebra $A$ is a  CB-algebra if, and only if, $A$ is an antiassociative algebra \cite{cb-algebras}. García-Martínez and  Van der Linden proved that
 an algebraically coherent variety of anti-commutative algebras is either a variety of Lie algebras or a variety of antiassociative algebras \cite{linden}.
They are also related to anticommutative $g$-alternative algebras \cite{CTL019}. 
Moreover, given an antiassociative algebra $A$, the new algebra $A^{(+)}$ is a Mock-Lie algebra,
which is a Jordan algebra of nilindex 3.
A Dual Mock-Lie algebra  is an anticommutative antiassociative algebra  \cite{Zusmanovich}.
Algebraic and geometric classifications of small dimensional dual Mock-Lie algebras are given in \cite{ckls}.
Another entertaining fact about these algebras is that  Markl and Remm showed the non-Koszulness of the corresponding operad for antiassociative algebras  \cite{MR014}, hence the deformation cohomology differs from the “standard” one, more precisely, the deformation cohomology coincides with the triple cohomology. Futhermore, antiassocitive algebras  playing an important  role in gauge theory, in particular, nonlinear gauge theory for two or more spin-2 fields \cite{Stephen02}.

Since  any product involving four or more elements of an antiassociative algebra is identically zero, (see, \cite[Lemma 1.1]{xyzw=0}), we can describe this variety in two  subsets: algebras that satisfies the identity $xyz=0$, called 2-step nilpotent  algebras, and algebras that do not satisfies the identity $xyz=0$, called non-2-step nilpotent algebras. Note that 2-step nilpotent algebras are contained in the intersection of varieties of algebras defined by families of polynomial identities of degree $3$, for example, Leibniz, associative and Zinbiel algebras. Then, the classification of antiassociative  2-step nilpotent algebras of dimension up to $4$ can be extracted from the list of algebras presented in \cite{kppv}. 
The algebraic classification of all $5$-dimensional $2$-step nilpotent algebras is a wild problem and it is not concluded until now, 
but there is a geometric classification of these algebras \cite{ikp20}, 
which will be used in the present work for a geometric classification of $5$-dimensional antiassociative algebras. 
The method for classifying antiassociative algebras is based on the calculation of central extensions of nilpotent algebras of smaller dimensions from the same variety.
The algebraic study of central extensions of   algebras has been an important topic for years \cite{  calferk,klp20,hac16,  ss78}.
First, Skjelbred and Sund used central extensions of Lie algebras to obtain a classification of nilpotent Lie algebras  \cite{ss78}.
Note that the Skjelbred-Sund method of central extensions is an important tool in the classification of nilpotent algebras.
Using the same method,  
small dimensional nilpotent 
(terminal  \cite{kkp20},  Lie  \cite{ degr2},   anticommutative  \cite{kkl20}, dual Mock-Lie   \cite{ckls}) algebras 
and some others have been described.
Our main results related to the algebraic classification of cited varieties are summarized below.

\begin{theoremA}
Up to isomorphism, there is only $1$ complex  $3$-dimensional  (non-2-step nilpotent) antiassociative   algebras, 
described explicitly  in  section \ref{4dim}.
Up to isomorphism, there are only $3$ complex  $4$-dimensional  (non-2-step nilpotent) antiassociative   algebras, 
described explicitly  in  section \ref{4dim}.
Up to isomorphism, there are infinitely many isomorphism classes of  
complex  $5$-dimensional   (non-2-step nilpotent) antiassociative   algebras, 
described explicitly  in  section \ref{5dim} in terms of 
$4$ one-parameter families and 
$26$ additional isomorphism classes.

\end{theoremA}

 The degenerations between the (finite-dimensional) algebras from a certain variety $\mathfrak{V}$ defined by a set of identities have been actively studied in the past decade.
The description of all degenerations allows one to find the so-called rigid algebras and families of algebras, i.e. those whose orbit closures under the action of the general linear group form irreducible components of $\mathfrak{V}$
(with respect to the Zariski topology). 
We list here some varieties in which the rigid algebras of the varieties of
all $4$-dimensional Leibniz algebras,
all $4$-dimensional nilpotent terminal algebras \cite{kkp20},
all $4$-dimensional nilpotent commutative algebras,
all $6$-dimensional nilpotent binary Lie algebras,
all $6$-dimensional nilpotent anticommutative algebras \cite{kkl20},
all $8$-dimensional dual Mock Lie algebras \cite{ckls}
have been found.
A full description of degenerations has been obtained  
for $2$-dimensional algebras, 
for $4$-dimensional Lie algebras in \cite{BC99},
for $4$-dimensional Zinbiel and  $4$-dimensional nilpotent Leibniz algebras in \cite{kppv},
for $5$-dimensional nilpotent Lie algebras in \cite{GRH},  
for $8$-dimensional $2$-step nilpotent anticommutative algebras \cite{ale3},
for $(n+1)$-dimensional $n$-Lie algebras, and some others.
Our main results related to the geometric classification of cited varieties are summarized below. 
\begin{theoremB}
The variety of complex $4$-dimensional  antiassociative algebras has 
dimension {\it 12 }  and it has 
3 irreducible components (in particular, there is only one  rigid algebra in this variety). 
The variety of complex $5$-dimensional   antiassociative  algebras has 
dimension {\it 24 }  and it has 
$8$ irreducible components (in particular, there are only $4$ rigid algebras in this variety). 

\end{theoremB}

\section{The algebraic classification of antiassociative  algebras}
\subsection{Method of classification of nilpotent algebras}
Throughout this paper, we use the notations and methods well written in \cite{  hac16},
which we have adapted for the antiassociative case with some modifications.
Further in this section we give some important definitions.

Let $({\bf A}, \cdot)$ be a complex  antiassociative   algebra 
and $\mathbb V$ be a complex  vector space. The $\mathbb C$-linear space ${\rm Z^{2}}\left(
\bf A,\mathbb V \right) $ is defined as the set of all  bilinear maps $\theta  \colon {\bf A} \times {\bf A} \longrightarrow {\mathbb V}$ such that
$$ \theta(xy,z)+\theta(x,yz)=0.$$
These elements will be called {\it cocycles}. 

For a linear map $f$ from $\bf A$ to  $\mathbb V$, if we define $\delta f\colon {\bf A} \times
{\bf A} \longrightarrow {\mathbb V}$ by $\delta f  (x,y ) =f(xy )$, then $\delta f\in {\rm Z^{2}}\left( {\bf A},{\mathbb V} \right) $. We define ${\rm B^{2}}\left({\bf A},{\mathbb V}\right) =\left\{ \theta =\delta f\ : f\in {\rm Hom}\left( {\bf A},{\mathbb V}\right) \right\} $.
We define the {\it second cohomology space} ${\rm H^{2}}\left( {\bf A},{\mathbb V}\right) $ as the quotient space ${\rm Z^{2}}
\left( {\bf A},{\mathbb V}\right) \big/{\rm B^{2}}\left( {\bf A},{\mathbb V}\right) $.

\

Let $\operatorname{Aut}({\bf A}) $ be the automorphism group of  ${\bf A} $ and let $\phi \in \operatorname{Aut}({\bf A})$. For $\theta \in
{\rm Z^{2}}\left( {\bf A},{\mathbb V}\right) $ define  the action of the group $\operatorname{Aut}({\bf A}) $ on ${\rm Z^{2}}\left( {\bf A},{\mathbb V}\right) $ by $\phi \theta (x,y)
=\theta \left( \phi \left( x\right) ,\phi \left( y\right) \right) $.  It is easy to verify that
 ${\rm B^{2}}\left( {\bf A},{\mathbb V}\right) $ is invariant under the action of $\operatorname{Aut}({\bf A}).$  
 So, we have an induced action of  $\operatorname{Aut}({\bf A})$  on ${\rm H^{2}}\left( {\bf A},{\mathbb V}\right)$.

\

Let $\bf A$ be a  antiassociative  algebra of dimension $m$ over  $\mathbb C$ and ${\mathbb V}$ be a $\mathbb C$-vector
space of dimension $k$. For the bilinear map $\theta$, define on the linear space ${\bf A}_{\theta } = {\bf A}\oplus {\mathbb V}$ the
bilinear product `` $\left[ -,-\right] _{{\bf A}_{\theta }}$'' by $\left[ x+x^{\prime },y+y^{\prime }\right] _{{\bf A}_{\theta }}=
 xy +\theta(x,y) $ for all $x,y\in {\bf A},x^{\prime },y^{\prime }\in {\mathbb V}$.
The algebra ${\bf A}_{\theta }$ is called a $k$-{\it dimensional central extension} of ${\bf A}$ by ${\mathbb V}$. One can easily check that ${\bf A_{\theta}}$ is a  antiassociative
algebra if and only if $\theta \in {\rm Z^2}({\bf A}, {\mathbb V})$.

Call the
set $\operatorname{Ann}(\theta)=\left\{ x\in {\bf A}:\theta \left( x, {\bf A} \right)+ \theta \left({\bf A} ,x\right) =0\right\} $
the {\it annihilator} of $\theta $. We recall that the {\it annihilator} of an  algebra ${\bf A}$ is defined as
the ideal $\operatorname{Ann}(  {\bf A} ) =\left\{ x\in {\bf A}:  x{\bf A}+ {\bf A}x =0\right\}$. Observe
 that
$\operatorname{Ann}\left( {\bf A}_{\theta }\right) =(\operatorname{Ann}(\theta) \cap\operatorname{Ann}({\bf A}))
 \oplus {\mathbb V}$.

\

The following result shows that every algebra with a non-zero annihilator is a central extension of a smaller-dimensional algebra.

\begin{lemma}
Let ${\bf A}$ be an $n$-dimensional  antiassociative algebra such that $\dim (\operatorname{Ann}({\bf A}))=m\neq0$. Then there exists, up to isomorphism, a unique $(n-m)$-dimensional  antiassociative  algebra ${\bf A}'$ and a bilinear map $\theta \in {\rm Z^2}({\bf A'}, {\mathbb V})$ with $\operatorname{Ann}({\bf A'})\cap\operatorname{Ann}(\theta)=0$, where $\mathbb V$ is a vector space of dimension m, such that ${\bf A} \cong {{\bf A}'}_{\theta}$ and
 ${\bf A}/\operatorname{Ann}({\bf A})\cong {\bf A}'$.
\end{lemma}

\begin{proof}
Let ${\bf A}'$ be a linear complement of $\operatorname{Ann}({\bf A})$ in ${\bf A}$. Define a linear map $P \colon {\bf A} \longrightarrow {\bf A}'$ by $P(x+v)=x$ for $x\in {\bf A}'$ and $v\in\operatorname{Ann}({\bf A})$, and define a multiplication on ${\bf A}'$ by $[x, y]_{{\bf A}'}=P(x y)$ for $x, y \in {\bf A}'$.
For $x, y \in {\bf A}$, we have
\[P(xy)=P((x-P(x)+P(x))(y- P(y)+P(y)))=P(P(x) P(y))=[P(x), P(y)]_{{\bf A}'}. \]

Since $P$ is a homomorphism $P({\bf A})={\bf A}'$ is an  antiassociative algebra and
 ${\bf A}/\operatorname{Ann}({\bf A})\cong {\bf A}'$, which gives us the uniqueness. Now, define the map $\theta \colon {\bf A}' \times {\bf A}' \longrightarrow\operatorname{Ann}({\bf A})$ by $\theta(x,y)=xy- [x,y]_{{\bf A}'}$.
  Thus, ${\bf A}'_{\theta}$ is ${\bf A}$ and therefore $\theta \in {\rm Z^2}({\bf A'}, {\mathbb V})$ and $\operatorname{Ann}({\bf A'})\cap\operatorname{Ann}(\theta)=0$.
\end{proof}

\

\begin{definition}
Let ${\bf A}$ be an algebra and $I$ be a subspace of $\operatorname{Ann}({\bf A})$. If ${\bf A}={\bf A}_0 \oplus I$
then $I$ is called an {\it annihilator component} of ${\bf A}$.
A central extension of an algebra $\bf A$ without annihilator component is called a {\it non-split central extension}.
\end{definition}

Our task is to find all central extensions of an algebra $\bf A$ by
a space ${\mathbb V}$.  In order to solve the isomorphism problem we need to study the
action of $\operatorname{Aut}({\bf A})$ on ${\rm H^{2}}\left( {\bf A},{\mathbb V}
\right) $. To do that, let us fix a basis $e_{1},\ldots ,e_{s}$ of ${\mathbb V}$, and $
\theta \in {\rm Z^{2}}\left( {\bf A},{\mathbb V}\right) $. Then $\theta $ can be uniquely
written as $\theta \left( x,y\right) =
\displaystyle \sum_{i=1}^{s} \theta _{i}\left( x,y\right) e_{i}$, where $\theta _{i}\in
{\rm Z^{2}}\left( {\bf A},\mathbb C\right) $. Moreover, $\operatorname{Ann}(\theta)=\operatorname{Ann}(\theta _{1})\cap\operatorname{Ann}(\theta _{2})\cap\ldots \cap\operatorname{Ann}(\theta _{s})$. Furthermore, $\theta \in
{\rm B^{2}}\left( {\bf A},{\mathbb V}\right) $\ if and only if all $\theta _{i}\in {\rm B^{2}}\left( {\bf A},
\mathbb C\right) $.
It is not difficult to prove (see \cite[Lemma 13]{hac16}) that given a  antiassociative algebra ${\bf A}_{\theta}$, if we write as
above $\theta \left( x,y\right) = \displaystyle \sum_{i=1}^{s} \theta_{i}\left( x,y\right) e_{i}\in {\rm Z^{2}}\left( {\bf A},{\mathbb V}\right) $ and 
$\operatorname{Ann}(\theta)\cap \operatorname{Ann}\left( {\bf A}\right) =0$, then ${\bf A}_{\theta }$ has an
annihilator component if and only if $\left[ \theta _{1}\right] ,\left[
\theta _{2}\right] ,\ldots ,\left[ \theta _{s}\right] $ are linearly
dependent in ${\rm H^{2}}\left( {\bf A},\mathbb C\right) $.

Let ${\mathbb V}$ be a finite-dimensional vector space over $\mathbb C$. The {\it Grassmannian} $G_{k}\left( {\mathbb V}\right) $ is the set of all $k$-dimensional
linear subspaces of $ {\mathbb V}$. Let $G_{s}\left( {\rm H^{2}}\left( {\bf A},\mathbb C\right) \right) $ be the Grassmannian of subspaces of dimension $s$ in
${\rm H^{2}}\left( {\bf A},\mathbb C\right) $. There is a natural action of $\operatorname{Aut}({\bf A})$ on $G_{s}\left( {\rm H^{2}}\left( {\bf A},\mathbb C\right) \right) $, given by $\phi W=\left\langle \left[ \phi \theta _{1}\right]
,\left[ \phi \theta _{2}\right] ,\dots,\left[ \phi \theta _{s}\right]
\right\rangle $,  for each 
 $\phi \in \operatorname{Aut}({\bf A})$ and each  $W=\left\langle
\left[ \theta _{1}\right] ,\left[ \theta _{2}\right] ,\dots,\left[ \theta _{s}
\right] \right\rangle \in G_{s}\left( {\rm H^{2}}\left( {\bf A},\mathbb C
\right) \right) $. We denote the orbit of $W\in G_{s}\left(
{\rm H^{2}}\left( {\bf A},\mathbb C\right) \right) $ under the action of $\operatorname{Aut}({\bf A})$ by $\operatorname{Orb}(W)$. Given
\[
W_{1}=\left\langle \left[ \theta _{1}\right] ,\left[ \theta _{2}\right] ,\dots,
\left[ \theta _{s}\right] \right\rangle ,W_{2}=\left\langle \left[ \vartheta
_{1}\right] ,\left[ \vartheta _{2}\right] ,\dots,\left[ \vartheta _{s}\right]
\right\rangle \in G_{s}\left( {\rm H^{2}}\left( {\bf A},\mathbb C\right)
\right),
\]
we have that if $W_{1}=W_{2}$, then $ \bigcap\limits_{i=1}^{s}\operatorname{Ann}(\theta _{i})\cap \operatorname{Ann}\left( {\bf A}\right) = \bigcap\limits_{i=1}^{s}
\operatorname{Ann}(\vartheta _{i})\cap\operatorname{Ann}( {\bf A}) $, and therefore we can introduce
the set
\[
{\bf T}_{s}({\bf A}) =\left\{ W=\left\langle \left[ \theta _{1}\right] ,
\left[ \theta _{2}\right] ,\dots,\left[ \theta _{s}\right] \right\rangle \in
G_{s}\left( {\rm H^{2}}\left( {\bf A},\mathbb C\right) \right) : \bigcap\limits_{i=1}^{s}\operatorname{Ann}(\theta _{i})\cap\operatorname{Ann}({\bf A}) =0\right\},
\]
which is stable under the action of $\operatorname{Aut}({\bf A})$.

\

Now, let ${\mathbb V}$ be an $s$-dimensional linear space and let us denote by
${\bf E}\left( {\bf A},{\mathbb V}\right) $ the set of all {\it non-split $s$-dimensional central extensions} of ${\bf A}$ by
${\mathbb V}$. By above, we can write
\[
{\bf E}\left( {\bf A},{\mathbb V}\right) =\left\{ {\bf A}_{\theta }:\theta \left( x,y\right) = \sum_{i=1}^{s}\theta _{i}\left( x,y\right) e_{i} \ \ \text{and} \ \ \left\langle \left[ \theta _{1}\right] ,\left[ \theta _{2}\right] ,\dots,
\left[ \theta _{s}\right] \right\rangle \in {\bf T}_{s}({\bf A}) \right\} .
\]
We also have the following result, which can be proved as in \cite[Lemma 17]{hac16}.

\begin{lemma}
 Let ${\bf A}_{\theta },{\bf A}_{\vartheta }\in {\bf E}\left( {\bf A},{\mathbb V}\right) $. Suppose that $\theta \left( x,y\right) =  \displaystyle \sum_{i=1}^{s}
\theta _{i}\left( x,y\right) e_{i}$ and $\vartheta \left( x,y\right) =
\displaystyle \sum_{i=1}^{s} \vartheta _{i}\left( x,y\right) e_{i}$.
Then the  antiassociative algebras ${\bf A}_{\theta }$ and ${\bf A}_{\vartheta } $ are isomorphic
if and only if
$$\operatorname{Orb}\left\langle \left[ \theta _{1}\right] ,
\left[ \theta _{2}\right] ,\dots,\left[ \theta _{s}\right] \right\rangle =
\operatorname{Orb}\left\langle \left[ \vartheta _{1}\right] ,\left[ \vartheta
_{2}\right] ,\dots,\left[ \vartheta _{s}\right] \right\rangle .$$
\end{lemma}

This shows that there exists a one-to-one correspondence between the set of $\operatorname{Aut}({\bf A})$-orbits on ${\bf T}_{s}\left( {\bf A}\right) $ and the set of
isomorphism classes of ${\bf E}\left( {\bf A},{\mathbb V}\right) $. Consequently, we have a
procedure that allows us, given a  antiassociative algebra ${\bf A}'$ of
dimension $n-s$, to construct all non-split central extensions of ${\bf A}'$. This procedure is:

\begin{enumerate}
\item For a given  antiassociative algebra ${\bf A}'$ of dimension $n-s $, determine ${\rm H^{2}}( {\bf A}',\mathbb {C}) $, $\operatorname{Ann}({\bf A}')$ and $\operatorname{Aut}({\bf A}')$.

\item Determine the set of $\operatorname{Aut}({\bf A}')$-orbits on ${\bf T}_{s}({\bf A}') $.

\item For each orbit, construct the  antiassociative algebra associated with a
representative of it.
\end{enumerate}

\medskip

\subsubsection{Notations}
Let us introduce the following notations. Let ${\bf A}$ be a nilpotent algebra with
a basis $\{e_{1},e_{2}, \ldots, e_{n}\}.$ Then by $\Delta_{ij}$\ we will denote the
bilinear form
$\Delta_{ij}:{\bf A}\times {\bf A}\longrightarrow \mathbb C$
with $\Delta_{ij}(e_{l},e_{m}) = \delta_{il}\delta_{jm}.$
The set $\left\{ \Delta_{ij}:1\leq i, j\leq n\right\}$ is a basis for the linear space of
bilinear forms on ${\bf A},$ so every $\theta \in
{\rm Z^2}({\bf A},\bf \mathbb V )$ can be uniquely written as $
\theta = \displaystyle \sum_{1\leq i,j\leq n} c_{ij}\Delta _{{i}{j}}$, where $
c_{ij}\in \mathbb C$.
We will use the following notations for our algebras:
$$\begin{array}{lll}
{\A}_{i,j}& \mbox{---}& j\mbox{th }i\mbox{-dimensional  $2$-step nilpotent algebra;} \\
{\Q}_{i,j}& \mbox{---}& j\mbox{th }i\mbox{-dimensional non-2-step nilpotent  antiassociative algebra.}
\end{array}$$

\subsection{The algebraic classification of  $4$-dimensional non-2-step nilpotent antiassociative algebras}\label{4dim}
Thanks to \cite{fkkv}, we have the classification of all 3-dimensional nilpotent algebras.
From this list we are choosing only $3$-dimensional antiassociative algebras: 

\begin{longtable}{ll llllllllllll}
${\A}_{3,1}$ &$:$& $e_1 e_1 = e_2$\\
${\A}_{3,2}$  &$:$& $e_1 e_2=e_3$  & $e_2 e_1=-e_3$   \\
${\A}_{3,3}$ &$:$&  $e_1 e_1 = e_3$  & $e_2 e_2=e_3$  \\
${\A}^{\lambda}_{3,4}$  &$:$& $e_1 e_1 = \lambda e_3$   & $e_2 e_1=e_3$  & $e_2 e_2=e_3$ \\
${\Q}_{3,1}$ &$:$& $e_1 e_1 = e_2$ & $e_1 e_2=e_3$ & $e_2 e_1= -e_3$
\end{longtable}

Note that the central extension with annihilator component of the non-$2$-step nilpotent algebra ${\Q}_{3,1}$, give us a four dimensional algebra denoted by ${\Q}_{4,1}$.

\
In the following table we give the description of the second cohomology space of  $3$-dimensional  $2$-step nilpotent  algebras.

\begin{longtable}{|lcl|}
                \hline
 
${\rm H}^2({\A}_{3,1})$& $=$& 
$\Big\langle 
 [\Delta_{12}]-[\Delta_{21}],[\Delta_{13}],[\Delta_{31}], [\Delta_{33}]
\Big\rangle$    \\ \hline

${\rm H}^2({\A}_{3,2})$& $=$& 
$\Big\langle 
[\Delta_{11}], [\Delta_{12}],  [ \Delta_{22}]
\Big\rangle$     \\ \hline

${\rm H}^2({\A}_{3,3})$& $=$& 
$ \Big\langle [\Delta_{11}] ,[\Delta_{12}],[ \Delta_{21}]  \Big\rangle$

    \\ \hline

${\rm H}^2({\A}^{\lambda}_{3,4})$& $=$& 
$\Big\langle 
[\Delta_{11}], [\Delta_{12}], [\Delta_{21}]

\Big\rangle$   \\ \hline
 \end{longtable}

 Since $\operatorname{Ann}({\A}_{3,i})=\langle e_3\rangle$ we have  ${\bf T}_{s}({\A}_{3,i})=\emptyset$, $i=2,3,4$, and, therefore, there is no central extension without annihilator component of these algebras.

\subsubsection{$1$-dimensional central extensions of ${\A}_{3,1}$}\label{1dim A301}\label{a31}
	Let us use the following notations:
\begin{longtable}{llll}
$\nabla_1 =[\Delta_{12}]-[\Delta_{21}],$& 
$\nabla_2 = [\Delta_{13}],$& 
$ \nabla_3 = [\Delta_{31}],$& 
$\nabla_4 = [\Delta_{33}].$
\end{longtable}
Take $\theta=\sum\limits_{i=1}^4\alpha_i\nabla_i\in {\rm H^2}({\A}_{3,1})$. Note that we can not have $\alpha_2=\alpha_3=\alpha_4=0$ or $\alpha_1=0$.
	The automorphism group of ${\A}_{3,1}$ consists of invertible matrices of the form
\begin{center}$	\phi=
	\begin{pmatrix}
	x &    0  &  0\\
	z &  x^2  &  r\\
	t &  0  &  y \\
	\end{pmatrix}.$\end{center}
	Since
	$$
	\phi^T\begin{pmatrix}
	0 &  \alpha_1 & \alpha_2\\
	-\alpha_1  &  0 & 0 \\
	\alpha_3&  0    & \alpha_4\\
	\end{pmatrix} \phi=	\begin{pmatrix}
	\alpha^\ast  &  \alpha_1^* & \alpha_2^*\\
	-\alpha_1^*  & 0 & 0 \\
	\alpha_3^*&  0    & \alpha_4^*
	\end{pmatrix},
	$$
	 we have that the action of ${\rm Aut} ({\A}_{3,1})$ on the subspace
$\Big\langle \sum\limits_{i=1}^4\alpha_i\nabla_i  \Big\rangle$
is given by
$\Big\langle \sum\limits_{i=1}^4\alpha_i^{*}\nabla_i\Big\rangle,$
where
\begin{longtable}{llll}
$\alpha^*_1=   \alpha_1x^3,$ &
$\alpha^*_2=    \alpha_1x r +  \alpha_2 xy+   \alpha_4yt,$ &
$\alpha^*_3=  \alpha_3 xy-\alpha_1xr  +  \alpha_4yt  ,$ &
$\alpha_4^*=   \alpha_4 y^2.$
\end{longtable}

Since $y\neq 0$,   $\alpha_4\neq 0$  if, and only if, $\alpha_4^{\ast}\neq 0$, we need to consider only the following cases:

\begin{enumerate}
     \item Consider $\alpha_4=0$.
     
     \begin{enumerate}
         \item If $\alpha_3=-\alpha_2$, by choosing $x=y=1$ and $r=-\frac{\alpha_2}{\alpha_1}$, we get the representative $\langle \nabla_1 \rangle$,
         that we do not consider since the obtained algebra would have $2$-dimensional annihilator.
          
         \item If $\alpha_3\neq -\alpha_2$, 
     by choosing $x=\frac{1}{\sqrt[3]{\alpha_1}}$, $y=\frac{\sqrt[3]{\alpha_1}}{\alpha_2+\alpha_3}$ and $r=\frac{\alpha_3}{(\alpha_2+\alpha_3)\sqrt[3]{\alpha_1^2}}$, we obtain the representative $\langle \nabla_1+  \nabla_2\rangle$.
     \end{enumerate}
     
      \item If $\alpha_4\neq 0$, 
            by choosing 
            $x=\alpha_1\alpha_4$, 
            $y=\alpha_1^2\alpha_4$, 
            $r=\frac{\alpha_1\alpha_4(\alpha_3-\alpha_2)}{2}$ and 
            $t=-\frac{\alpha_1(\alpha_2+\alpha_3)}{2},$
            we obtain the representative $\langle \nabla_1+  \nabla_4\rangle$.

\end{enumerate}

Summarizing, we have the following distinct orbits: 
\begin{center} 
$\langle  \nabla_1+\nabla_2 \rangle$ and 
$\langle \nabla_1 + \nabla_4\rangle$.
\end{center}
These orbits give us the following algebras:

 \begin{longtable}{llllll}

${\Q}_{4,2}$ &$:$& $ e_1e_1=e_2$ & $e_1 e_2 =e_4$ & $e_1 e_3 =e_4 $&  $e_2 e_1 =-e_4$  \\

${\Q}_{4,3}$ &$:$& $ e_1e_1=e_2$ & $e_1 e_2 =e_4$ &  $e_2 e_1 =-e_4$ & $e_3 e_3 = e_4$\\

    \end{longtable}
 \subsection{The algebraic classification of $5$-dimensional non-2-step nilpotent antiassociative algebras}\label{5dim}
It is easy to see that  
there are no non-2-step nilpotent $3$-dimensional central extensions of antiassociative algebras of dimension $2$.   The split central extension of $4$-dimensional algebras obtained in Section \ref{4dim} will be denoted, respectively, by ${\Q}_{5,1}$, ${\Q}_{5,2}$ and ${\Q}_{5,3}$.  Hence, to classify $5$-dimensional non-split antiassociative algebras, we need to find only $2$-dimensional central extensions of  ${\A}_{3,1}$
and $1$-dimensional central extensions of $4$-dimensional $2$-step nilpotent algebras. 

\subsubsection{$2$-dimensional central extensions of ${\A}_{3,1}$}
  	Let us use the  notation from subsection \ref{a31}.
Hence, consider the vector space in ${\bf T}_2(\A_{3,1})$ generated by the following two cocycles: 
 \begin{longtable}{l}$\theta_1=\alpha_1 \nabla_1+ \alpha_{2}\nabla_2+\alpha_{3}\nabla_3 + \alpha_4\nabla_4;$ \\
 $\theta_2=\beta_1 \nabla_1+ \beta_{2}\nabla_2+\beta_{3}\nabla_3 + \beta_4\nabla_4$. 
    \end{longtable}
By the results obtained in the subsection \ref{1dim A301}, we can consider \begin{center}$\theta_1 \in 
\{ \nabla_1,  \nabla_1+  \nabla_2, \nabla_1+ \nabla_4\}.$
\end{center}
Therefore, we have:
\begin{enumerate}

\item $\theta_1=\nabla_1$. Then, we can suppose $\beta_1=0$.
\begin{enumerate}
    \item If $\beta_4=0$, we have the following cases:
    \begin{enumerate}
        \item If $\beta_2=0$, we have the representative $\langle \nabla_1, \nabla_3 \rangle$.
        \item If $\beta_2 \neq 0$, we have the family of  representatives  $\langle \nabla_1, \nabla_2+\lambda \nabla_3 \rangle$.
    \end{enumerate}
    \item If $\beta_4 \neq 0$, we can suppose $\beta_4=1$.
    \begin{enumerate}
     \item If $\beta_2 =\beta_3$, consider $x=y=1$ and $t=-\beta_2$, then, we obtain the representative   $\langle \nabla_1, \nabla_4 \rangle$.
        \item If $\beta_2 \neq \beta_3$, consider $x=\frac{1}{\beta_2-\beta_3}$, $y=1$ and $t=\frac{\beta_3}{\beta_3-\beta_2}$, then, we obtain the representative   $\langle \nabla_1, \nabla_2+ \nabla_4 \rangle$.
    \end{enumerate}
\end{enumerate}
    \item  \label{caso 1 a32 dim5} $\theta_1= \nabla_1+\nabla_2.$
    Note that we can consider $\beta_1=0$.
    \begin{enumerate}
        \item Suppose $\beta_4=0$.
        
        \begin{enumerate}
        \item If $\beta_2=0$ and $\beta_3 \neq 0$, by choosing  $x=y=-r=1$,  we obtain the representative $ \la\nabla_1, \nabla_3\ra$.
        
        \item If  $\beta_2 \neq 0$ and $\beta_3=0$, we have the representative $\langle \nabla_1,\nabla_2\rangle$.
        \item Suppose $\beta_2\beta_3\neq 0$.
        \begin{enumerate}
             \item If $\beta_2=-\beta_3$, we obtain the   representative  $\langle \nabla_1+\nabla_2,\nabla_2-\nabla_3\rangle$.
          \item If $\beta_2 \neq -\beta_3$, by choosing $x=y=1$ and $r= -\frac{\beta_3}{\beta_2+\beta_3}$, we obtain the family of representatives  $\langle \nabla_1,\nabla_2+\lambda\nabla_3\rangle$, with $\lambda\neq 0$. 
        \end{enumerate}


        \end{enumerate}
        \item  If $\beta_4\neq 0$, we can suppose $\beta_4=1$.
    \begin{enumerate}
 
    \item\label{1bi} If $\beta_3=\beta_2$, by choosing $x=y=1$ and $t=-\beta_2$, we obtain the representative 
    $ \la\nabla_1+\nabla_2,  \nabla_4\ra$.

        
        \item\label{1bii} If $ \beta_3\neq\beta_2$, by choosing  $x=\beta_3-\beta_2$, $y=(\beta_3-\beta_2)^2$ and $t=-\beta_2(\beta_3-\beta_2)$, we obtain the representative
           $ \la\nabla_1+\nabla_2,\nabla_3+\nabla_4\ra $.
    \end{enumerate}
    \end{enumerate}
      \item  $\theta_1 = \nabla_1+  \nabla_4$. 
      Note that we can consider $\beta_4=0$.
          \begin{enumerate}
        \item Suppose $\beta_1=0$.
        \begin{enumerate}
            \item If $\beta_2=0$, we have the representative $ \la\nabla_1+\nabla_4, \nabla_3\ra$.
            
            \item     If $\beta_2\neq 0$, we have the family of representatives
        $\la\nabla_1+\nabla_4, \nabla_2 + \lambda\nabla_3\ra$.
        \end{enumerate}

        \item If $ \beta_1\neq 0$, we can suppose $\beta_1 =1$.
        \begin{enumerate}
       
            \item Suppose $\beta_2=\beta_3$. 
            \begin{enumerate}
                \item If $\beta_2=\beta_3=0$, we have the representative $\langle \nabla_1,\nabla_4\rangle$.
                \item Otherwise, by choosing $x=1$, $y=\frac{1}{2\beta_2}$ and $r=\frac{1}{2}$, we obtain  the case \ref{1bi}.
            \end{enumerate}
            \item If $\beta_2\neq \beta_3$, we have:
            \begin{enumerate}
             \item If  $\beta_3=-\beta_2\neq  0$, by choosing $x=-1$, $r=-2\beta_2^2$, $y =2\beta_2$ and $t=\beta_2$, we obtain the representative $\la \nabla_1,\nabla_2+\nabla_4\ra$.
                \item If  $\beta_2\neq -\beta_3$, by choosing $x=1$, $y=\frac{1}{\beta_2+\beta_3}$ and $r=\frac{\beta_3}{\beta_2+\beta_3}$, we obtain, again,  the case \ref{1bii}.
                
            \end{enumerate}
        \end{enumerate}
        
    \end{enumerate}
\end{enumerate}
 Summarizing, we have the following distinct orbits: 
\begin{center} 
$\la\nabla_1, \nabla_3\ra$, $\la\nabla_1, \nabla_2+\lambda\nabla_3\ra$, $\la\nabla_1, \nabla_2+\nabla_4\ra$, $\la\nabla_1, \nabla_4\ra$ , $\la\nabla_1+\nabla_2, \nabla_2-\nabla_3\ra$,   $ \la\nabla_1+\nabla_2,  \nabla_4\ra$,      $ \la\nabla_1+\nabla_2,\nabla_3+\nabla_4\ra $, $ \la\nabla_1+\nabla_4, \nabla_3\ra$ and  $\la\nabla_1+\nabla_4, \nabla_2 + \lambda\nabla_3\ra$.
\end{center}
A standard computation shows that these  orbits are pairwise distinct. Thus, these orbits give us the following algebras:
\begin{longtable}{lllllllll}
$\Q_{5,4}$&$:$&$ e_1  e_1=e_2$&$  e_1  e_2=e_4$&$   e_2  e_1=-e_4 $&$  e_3  e_1 =  e_5$\\ 
$\Q_{5,5}^{\lambda}$&$ :$&$  e_1  e_1=e_2$&$  e_1  e_2=e_4$&$  e_1  e_3 =e_5$&$  e_2  e_1=-e_4 $&$  e_3  e_1 = \lambda e_5$\\
$\Q_{5,6}$&$ :$&$ e_1  e_1=e_2$&$  e_1e_2=e_4$&$  e_1e_3=e_5$&$  e_2e_1=-e_4$&$ e_3e_3= e_5$\\
$\Q_{5,7}$&$ :$&$ e_1  e_1=e_2$&$  e_1e_2=e_4$&$  e_2e_1=-e_4$&$ e_3e_3= e_5$\\
$\Q_{5,8}$&$ : $&$ e_1  e_1=e_2$&$ e_1e_2=e_4$&$  e_1e_3=e_4+e_5$&$  e_2e_1=-e_4$&$  e_3e_1=-e_5$&\\
$\Q_{5,9}$&$  :$&$ e_1  e_1=e_2$&$  e_1e_2=e_4$&$e_1e_3=e_4$&$  e_2e_1=-e_4$&$  e_3e_3=e_5$&\\
$\Q_{5,10}$&$ : $&$ e_1  e_1=e_2$&$ e_1e_2=e_4$&$  e_1e_3=e_4$&$  e_2e_1=-e_4$&$  e_3e_1=e_5$&$  e_3e_3=e_5$\\
$\Q_{5,11}$&$  :$&$ e_1  e_1=e_2$&$  e_1e_2=e_4$&$  e_2e_1=-e_4$&$  e_3e_1=e_5$&$e_3e_3=e_4$\\
$\Q_{5,12}^{\lambda}$&$  :$&$  e_1  e_1=e_2$&$  e_1e_2=e_4$&$  e_1e_3=e_5$&$  e_2e_1=-e_4$&$  e_3e_1=\lambda e_5$&$ e_3e_3=e_4$& \\

\end{longtable}

\subsubsection{$1$-dimensional central extensions of $4$-dimensional $2$-step nilpotent algebras}
Thanks to \cite{kppv}, we have the classification of all $4$-dimensional nilpotent algebras.
From this list we are choosing only $4$-dimensional $2$-step nilpotent algebras:

 \begin{longtable}{lllllllllll}

			$\A_{4,1}$ &$:$& $e_1 e_1 =e_2$ \\
			  
		$\A_{4,2}$ &$:$& $e_1e_1 = e_2$ & $ e_3 e_3 = e_4$ \\
		
		$\A_{4,3}$ &$:$& $e_1 e_2 = e_3$ & $e_2 e_1 =-e_3$ &\\
		
		$\A_{4,4}^{\lambda}$ &$:$& $e_1 e_1 = e_3$ & $ e_1 e_2 = e_3$ & $ e_2 e_2 = \lambda e_3$ &  \\
		
		$\A_{4,5}$ &$:$& $e_1 e_1 = e_3$ & $e_1 e_2 = e_3$ & $ e_2 e_1 = e_3 $   \\
		
	 $\A_{4,6}$ &$:$& $e_1 e_2 = e_4$ & $  e_3 e_1 = e_4$ \\
	   
	 $\A_{4,7}$ &$:$& $e_1 e_2 = e_3$ & $e_2 e_1 = e_4$ & $e_2e_2 =-e_3$ \\
	  
	 $\A_{4,8}^{\lambda}$ &$:$& $e_1 e_1 = e_3$ & $  e_1 e_2 = e_4$ & $  e_2 e_1 =-\lambda e_3$ & $ e_2 e_2 =-e_4$ &  \\
	 
	 $\A_{4,9}^{\lambda}$ &$:$& $e_1 e_1 = e_4$ & $ e_1 e_2 = \lambda e_4$ & $  e_2 e_1 =-\lambda e_4$ & $ e_2 e_2 = e_4$ & $e_3 e_3 = e_4$ &   \\

	 $\A_{4,10}$ &$:$& $e_1 e_2 = e_4$ & $ e_1 e_3 = e_4$ & $ e_2e_1 =-e_4$ & $ e_2  e_2 = e_4$ & $ e_3 e_1 = e_4  $ \\
	 
	 $\A_{4,11}$ &$:$& $e_1 e_1 = e_4$ & $e_1e_2 = e_4$ & $e_2e_1 =-e_4$ & $  e_3 e_3 = e_4 $ \\
	 $\A_{4,12}$ &$:$& $ e_1 e_2 = e_3$ & $ e_2 e_1 = e_4 $ \\
	 $\A_{4,13}$ &$:$& $ e_1 e_1 = e_4$ & $  e_1 e_2 = e_3$ & $ e_2e_1 =-e_3$ & $e_2 e_2 = 2e_3 + e_4$ \\
	$\A_{4,14}^{\lambda}$ &$:$& $e_1 e_2 = e_4$ & $ e_2 e_1 = \lambda e_4$ & $e_2 e_2 = e_3$  \\
 $\A_{4,15}$ &$:$& $ e_1 e_2 = e_4$ & $  e_2 e_1 =-e_4$ & $ e_3 e_3 = e_4$ \\

    \end{longtable}

In the following table we give the description of the second cohomology space of  $4$-dimensional $2$-step nilpotent algebras.

\begin{longtable}{|lcl|}
                \hline
 
${\rm H}^2({\A}_{4,1})$& $=$& 
$\Big\langle 
 [\Delta_{12}]-[\Delta_{21}],  [\Delta_{13}], [\Delta_{14}], [\Delta_{31}], [\Delta_{33}],[\Delta_{34}], [\Delta_{41}], [\Delta_{43}], [\Delta_{44}]
\Big\rangle$    \\ \hline

${\rm H}^2({\A}_{4,2})$& $=$& 
$\Big\langle 
[\Delta_{12}]-[\Delta_{21}],  [\Delta_{13}],  [\Delta_{31}],[\Delta_{34}]-[\Delta_{43}]
\Big\rangle$     \\ \hline

${\rm H}^2({\A}_{4,3})$& $=$& 
$ \Big\langle [\Delta_{11}],  [\Delta_{12}], [\Delta_{14}], [\Delta_{22}], [\Delta_{24}], [\Delta_{41}], [\Delta_{42}],  [\Delta_{44}] \Big\rangle$

    \\ \hline

${\rm H}^2({\A}^{\lambda}_{4,4})$& $=$& 
$\Big\langle 
 [\Delta_{12}], [\Delta_{14}], [\Delta_{21}], [\Delta_{22}], [\Delta_{24}], [\Delta_{41}], [\Delta_{42}], [\Delta_{44}]
\Big\rangle$   \\ \hline

${\rm H}^2({\A}_{4,5})$& $=$& 
$ \Big\langle [ \Delta_{11}], [ \Delta_{12}], [\Delta_{14}], [\Delta_{22}],[\Delta_{24}], [\Delta_{41}], [\Delta_{42}], [\Delta_{44}] \Big\rangle$

    \\ \hline
    
    ${\rm H}^2({\A}_{4,6})$& $=$& 
$ \Big\langle  [\Delta_{11}],  [\Delta_{12}], [\Delta_{13}], [\Delta_{21}], [\Delta_{22}],[\Delta_{23}],[ \Delta_{32}], [\Delta_{33}] \Big\rangle$

    \\ \hline
    
       ${\rm H}^2({\A}_{4,7})$& $=$& 
$ \Big\langle [\Delta_{11}], [\Delta_{12}], [\Delta_{14}]-[\Delta_{24}]-[\Delta_{31}] \Big\rangle$

    \\ \hline
${\rm H}^2({\A}^{\lambda}_{4,8})$& $=$& 
$\Big\langle 
[\Delta_{21}], [\Delta_{12}]
\Big\rangle$   \\ \hline
${\rm H}^2({\A}^{\lambda}_{4,9})$& $=$& 
$\Big\langle 
 [\Delta_{11}], [ \Delta_{12}], [\Delta_{13}], [\Delta_{21}], [\Delta_{22}],[\Delta_{23}], [\Delta_{31}],[\Delta_{32}]
\Big\rangle$   \\ \hline
      ${\rm H}^2({\A}_{4,10})$& $=$& 
$ \Big\langle  [\Delta_{11}], [ \Delta_{12}], [\Delta_{13}], [\Delta_{21}], [\Delta_{22}],[\Delta_{23}], [\Delta_{32}],[\Delta_{33}] \Big\rangle$ 
    \\ \hline
       ${\rm H}^2({\A}_{4,11})$& $=$& 
$ \Big\langle  [\Delta_{11}], [ \Delta_{12}], [\Delta_{13}], [\Delta_{21}], [\Delta_{22}],[\Delta_{23}], [\Delta_{31}],[\Delta_{32}]\Big\rangle$ 
    \\ \hline
       ${\rm H}^2({\A}_{4,12})$& $=$& 
$ \Big\langle [\Delta_{11}], [\Delta_{14}]-[\Delta_{31}],  [\Delta_{22}], [\Delta_{23}]-[\Delta_{42}]\Big\rangle$ 
    \\ \hline
       ${\rm H}^2({\A}_{4,13})$& $=$& 
$ \Big\langle[\Delta_{11}] , [\Delta_{12}]\Big\rangle$ 
    \\ \hline
    ${\rm H}^2({\A}^{\lambda}_{4,14})$& $=$& 
$\Big\langle 
 [\Delta_{11}],[ \Delta_{21}],[ \Delta_{13}]-\lambda^2[\Delta_{31}]+\lambda [\Delta_{24}]-[\Delta_{42}], [\Delta_{23}]-[\Delta_{32}]
\Big\rangle$   \\ \hline
  ${\rm H}^2({\A}_{4,15})$& $=$& 
$ \Big\langle   [\Delta_{11}], [ \Delta_{12}], [\Delta_{13}], [\Delta_{21}], [\Delta_{22}],[\Delta_{23}], [\Delta_{31}], [\Delta_{32}]\Big\rangle$ 
    \\ \hline
 \end{longtable}
 
 Note that we can just consider central extensions of  ${\A}_{4,1}$, ${\A}_{4,2}$, ${\A}_{4,7}$, ${\A}_{4,12}$ and ${\A}_{4,14}^{\lambda}$, since ${\bf T}_{s}({\A}_{4,i})=\emptyset$, in the other cases.

\subsubsection{Central extensions of ${\A}_{4,1}$}
Let us use the following notations:
\begin{longtable}{lllll}
$\nabla_1 =[\Delta_{12}]-[\Delta_{21}],$& 
$\nabla_2 =  [\Delta_{13}],$& 
$ \nabla_3=[\Delta_{14}],$& 
$ \nabla_4=[\Delta_{31}],$& 
$\nabla_5=[\Delta_{33}],$\\
$\nabla_6=[\Delta_{34}],$& 
$\nabla_7=[\Delta_{41}],$& 
$\nabla_8=[\Delta_{43}],$& 
$\nabla_9=[\Delta_{44}].$
\end{longtable}
Take $\theta=\sum\limits_{i=1}^9\alpha_i\nabla_i\in {\rm H^2}({\A}_{4,1})$.	The automorphism group of ${\A}_{4,1}$ consists of invertible matrices of the form
\begin{center}
    $\phi= \begin{pmatrix} x&0 &0 & 0\\r&x^{2}& s &t\\u&0 &y& v \\
   p& 0& q&z\end{pmatrix}.$
\end{center}
		Since
	$$
	\phi^T\begin{pmatrix}
	0 &  \alpha_1 & \alpha_2&\alpha_3\\
	-\alpha_1  &  0 & 0 &0\\
	\alpha_4&  0  &\alpha_5  & \alpha_6\\
	\alpha_7&0&\alpha_8&\alpha_9
	\end{pmatrix} \phi=\begin{pmatrix}
	\alpha^* &  \alpha_1^* & \alpha_2^*&\alpha_3^*\\
	-\alpha_1^* &  0 & 0 &0\\
	\alpha_4^*&  0  &\alpha_5^*  & \alpha_6^*\\
	\alpha_7^*&0&\alpha_8^*&\alpha_9^*
	\end{pmatrix},
	$$
	 we have that the action of ${\rm Aut} ({\A}_{4,1})$ on the subspace
$\Big\langle \sum\limits_{i=1}^9\alpha_i\nabla_i  \Big\rangle$
is given by
$\Big\langle \sum\limits_{i=1}^9\alpha_i^{*}\nabla_i\Big\rangle,$
where
\begin{longtable}{lcl}
$    \alpha_1^{*}$&$=$&$ \alpha_1x^3, $\\ 
$
\alpha_2^{*}$&$=$&$
\alpha_1 x s + \alpha_2 xy + \alpha_3xq + \alpha_5u y  +\alpha_6 uq  + 
 \alpha_8 y p  +\alpha_9 p q, $\\
 $
\alpha_3^{*}$&$=$&$\alpha_1x t+\alpha_2 x v +\alpha_3 x z + \alpha_5uv  + \alpha_6 u z + 
 \alpha_8 vp  + \alpha_9pz, $\\
 $ 
\alpha_4^{*}$&$=$&$ -\alpha_1x s + \alpha_4x y  + \alpha_5u y  +\alpha_6 y p  + \alpha_7x q  + 
\alpha_8 uq  + \alpha_9pq,$\\
$\alpha_5^{*}$&$=$&$\alpha_5 y^2  + (\alpha_6 + \alpha_8)y q  + \alpha_9q^2,  $\\ 
$
\alpha_6^{*}$&$=$&$\alpha_5yv  + \alpha_6y z  + \alpha_8vq  +\alpha_9qz, $\\
$
\alpha_7^{*}$&$=$&$ -\alpha_1 x t + \alpha_4x v  + \alpha_5uv  + \alpha_6pv  + \alpha_7x z + 
   \alpha_8uz  + \alpha_9pz, $\\ 
   $
\alpha_8^{*}$&$=$&$\alpha_5y v + \alpha_6qv + \alpha_8yz + \alpha_9qz,  $\\
$\alpha_9^{*}$&$=$&$ \alpha_5v^2  +(\alpha_6 + \alpha_8)vz + \alpha_9z^2. $
\end{longtable}
Since $\operatorname{Ann}(\A_{4,1})= \la e_2,e_3, e_4\ra $, we can not have $\alpha_1 = 0$. So, we can consider  $\alpha_1=1$. Similarly,  we can not have $\alpha_2=\alpha_4=\alpha_5=\alpha_6=\alpha_8=0$ and, also, $\alpha_3=\alpha_6=\alpha_7=\alpha_8=\alpha_9=0$.
\begin{remark}
By choosing $\phi$ given by $x=y=z=1$,  $s=\alpha_4$ and  $t=\alpha_7$, we obtain  \begin{center}
    $\phi \theta = \nabla_1 + (\alpha_2+\alpha_4)\nabla_2 + (\alpha_3+\alpha_7)\nabla_3+\alpha_5 \nabla_5+ \alpha_6\nabla_6+  \alpha_8\nabla_8+\alpha_9\nabla_9$. 
\end{center} So, we can consider $\alpha_4=\alpha_7=0$.  Moreover, considering $\phi$ given by $x=v=q=1,$ we see that  $\phi \theta $ permutes the coefficients $\alpha_2$ with $ \alpha_3$, $\alpha_5 $ with $\alpha_9$ and also $\alpha_6$ with $\alpha_8$. In other words, we have $(\alpha_2^{*}, \alpha_3^{*},\alpha_5^{*}, \alpha_{6}^{*},\alpha_8^{*}, \alpha_9^{*})= (\alpha_3, \alpha_2,\alpha_9, \alpha_8, \alpha_6, \alpha_5)$. 
\end{remark}
Therefore, we can consider only the following cases:
\begin{enumerate}
        \item $\theta=  \nabla_1 + \alpha_2 \nabla_2 + \alpha_3\nabla_3$. Note that $\alpha_2\alpha_3 \neq 0.$ By choosing  $x=y=z=1$, $v=\frac{1-\alpha_3}{\alpha_2}$ and $q=-\frac{\alpha_2}{\alpha_3}$, we obtain   $\la \nabla_1+\nabla_3 \ra$, that we do not consider since the obtained algebra would have $2$-dimensional annihilator.

            \item \label{caso 68 a42 =0} $\theta=  \nabla_1 + \alpha_2 \nabla_2 + \alpha_3\nabla_3 +\alpha_6 \nabla_6+\alpha_8\nabla_8$, with $(\alpha_6,\alpha_8)\neq (0,0)$.
            
            \begin{enumerate}
                \item If     $\alpha_8=-\alpha_6$  and  $\alpha_6\neq 0$, we have:
                    \begin{enumerate}
    
        \item  If  $\alpha_2 \neq 0$, then, by choosing   $x=1$, $y=\frac{1}{\alpha_{2}}$, $v=-\frac{\alpha_3}{\alpha_6}$ and $z=\frac{\alpha_2}{\alpha_6}$, we obtain the representative $ \la\nabla_1+\nabla_2 +  \nabla_6-\nabla_8\ra$.

\item If $\alpha_2=\alpha_3=0$,  by  choosing  $x=z=1$ and $y=\frac{1}{\alpha_6}$, we obtain the representative $ \la\nabla_1+ \nabla_6-\nabla_8\ra$.
    \end{enumerate}
                \item\label{caso 68 a42} If  $\alpha_8 \neq -\alpha_6$, note that we can suppose $\alpha_8\neq 0$. By choosing $x=v=1$, $s=u=-\frac{\alpha_3 }{\alpha_6 + \alpha_8}$, $t=-\frac{\alpha_2\alpha_6 }{\alpha_6 + \alpha_8}$, $p=-\frac{\alpha_2 }{\alpha_6 + \alpha_8}$ and $q=\frac{1}{\alpha_8}$,   we have the family of representatives  $\la \nabla_1+\nabla_6+\lambda\nabla_8\ra$.
            \end{enumerate}
      \item\label{123568} $\theta_1=  \nabla_1 + \alpha_2 \nabla_2 + \alpha_3\nabla_3 +\alpha_5 \nabla_5+\alpha_6\nabla_6+\alpha_8\nabla_8$, with $\alpha_5\neq 0$.            
 \begin{enumerate}
  \item Suppose that $\alpha_6=\alpha_8=0$.   
       Note that $\alpha_3 \neq 0$. So, by choosing $x=1$,  $y=\frac{1}{\sqrt{\alpha_5}}$ , $ q=-\frac{\alpha_2}{\alpha_3 \sqrt{\alpha_5}}$ and  $z=\frac{1}{\alpha_{3}}$, we obtain the representative $ \la \nabla_1+\nabla_3+ \nabla_5\ra$.
         \item  \label{caso a42 1 a6} Suppose $\alpha_6 \neq 0$ and $\alpha_8=0$.
    By choosing  $x=y=1$,  $z=\frac{1}{\alpha_6}$,  $q=-\frac{\alpha_5}{\alpha_6}$, $u=-\frac{\alpha_3}{\alpha_6}$, $p=\frac{2\alpha_3\alpha_5-\alpha_2\alpha_6}{\alpha_6^2}$  and $s=\frac{\alpha_3\alpha_5-\alpha_2\alpha_6}{\alpha_6}$, we obtain the representative $\la \nabla_1+\nabla_6\ra$.

 \item If $\alpha_6=0$ and $\alpha_8\neq 0$, by choosing $x=q=1$, $ v=-\alpha_8$ and $z=\alpha_5$, we obtain again the case \ref{caso 68 a42}.
              
              \item \label{caso 568 a42 =0}   Suppose $\alpha_6\alpha_8\neq 0$.    
              \begin{enumerate}
                  \item If  $\alpha_8  \neq -\alpha_6$,
             by choosing  $x=q=z=1$ and  $v=-\frac{\alpha_6 + \alpha_8}{\alpha_5}$, we obtain again the case \ref{caso 68 a42}. 
             \item If $\alpha_8=-\alpha_6 $, we have two cases: \begin{enumerate}
                  \item If $\alpha_3=0$, 
              by choosing $x=1$, $y=\frac{1}{\sqrt{\alpha_5}}$, $z=\frac{\sqrt{\alpha_5}}{\alpha_6}$, $u=-\frac{\alpha_2}{2\alpha_5}$ and $s=-t=-\frac{\alpha_2}{2\sqrt{\alpha_5}}$, we obtain   the representative  $\la \nabla_1+\nabla_5+\nabla_6-\nabla_8\ra$.
              \item Otherwise, if $\alpha_3\neq 0$, by choosing $x=\frac{\alpha_3^2 \alpha_5}{\alpha_6^2}$, $y=\frac{\alpha_3^3 \alpha_5}{\alpha_6^3}$, $z=\frac{\alpha_3^3 \alpha_5^2}{\alpha_6^4}$, $u=-\frac{\alpha_2\alpha_3^2 }{2\alpha_6^2}$ and $s=-t=-\frac{\alpha_2\alpha_3^3 \alpha_5}{2\alpha_6^3}$, we obtain the representative $\la \nabla_1+\nabla_3+\nabla_5+\nabla_6-\nabla_8\ra$.
              \end{enumerate} 
              \end{enumerate}


  
          \end{enumerate}
           \item $\theta=  \nabla_1 + \alpha_2 \nabla_2 + \alpha_3\nabla_3 +\alpha_5 \nabla_5+\alpha_6\nabla_6+ \alpha_8\nabla_8 + \alpha_9\nabla_9$, with $\alpha_9\neq 0$.
              By choosing $x=2\alpha_9$, $y=v=1$, $q=-\frac{1 + \alpha_6 + \alpha_8 - 
  \sqrt{(\alpha_6 + \alpha_8)^2 - 
   4 \alpha_5 \alpha_9}}{2 \alpha_9}$ and $z=-\frac{ \alpha_6 + \alpha_8 - 
  \sqrt{(\alpha_6 + \alpha_8)^2 - 
   4 \alpha_5 \alpha_9}}{2 \alpha_9}$, we obtain the case   \ref{caso 68 a42 =0} or the case \ref{123568}, depending on the value of $\alpha_5^{\ast}$.
                      
                
\end{enumerate}
A standard computation shows that these  orbits are pairwise distinct.
Summarizing, we get the distinct orbits 
\begin{center}
$ \la \nabla_1+\nabla_3+ \nabla_5\ra$, $\la \nabla_1+\nabla_6+\lambda \nabla_8\ra$,  $ \la\nabla_1+\nabla_2 +  \nabla_6-\nabla_8\ra $,\\ $\la \nabla_1+ \nabla_3+\nabla_5+\nabla_6-\nabla_8\ra$ and $\la \nabla_1+\nabla_5+\nabla_6-\nabla_8\ra$,\end{center} 
which give us the following algebras:
\begin{longtable}{llllllllll}
$\Q_{5,13}$& $:$ & $e_1e_1=e_2$ & $ e_1e_2=e_5$& $e_1e_4=e_5$ & $e_2e_1=-e_5$ & $ e_3e_3= e_5$& &\\
$\Q_{5,14}^{\lambda}$& $:$ &  $e_1e_1=e_2$ &   $ e_1e_2=e_5$& $e_2e_1=-e_5$ & $ e_3e_4=e_5$ & $ e_4e_3=\lambda e_5 $\\
$\Q_{5,15}$ & $:$ &  $ e_1e_1=e_2$  & $ e_1e_2=e_5$&$  e_1e_3=e_5$& $e_2e_1=-e_5$  & $ e_3e_4=e_5$&$e_4e_3=-e_5$  \\
$\Q_{5,16}$& $:$ &  $ e_1e_1=e_2$  & $ e_1e_2=e_5$& $e_1e_4=e_5$& $e_2e_1=-e_5$  &$e_3e_3=e_5 $ &$ e_3e_4=e_5$ & $e_4e_3=-e_5$ \\
$\Q_{5,17}$& $:$ &  $ e_1e_1=e_2$  & $ e_1e_2=e_5$& $e_2e_1=-e_5$  &$e_3e_3=e_5 $ &$ e_3e_4=e_5$ & $e_4e_3=-e_5$ \\
\end{longtable}

\subsubsection{Central extensions of ${\A}_{4,2}$}

Let us use the following notations:
\begin{longtable}{llll}
$\nabla_1 = [\Delta_{12}]-[\Delta_{21}],$& 
$\nabla_2 = [\Delta_{13}],$& 
$\nabla_3 = [\Delta_{31}],$ & 
$\nabla_4 = [\Delta_{34}]-[\Delta_{43}].$
\end{longtable}
Take $\theta=\sum\limits_{i=1}^4\alpha_i\nabla_i\in {\rm H^2}({\A}_{4,2})$. Since $\operatorname{Ann}(\A_{4,2})=\langle e_2, e_4\rangle$, we can not have $\alpha_1\alpha_4=0.$ 
	The automorphism group of ${\A}_{4,2}$ consists of invertible matrices of the forms
\begin{center}$	\phi_1=
	\begin{pmatrix}
	x &    0  &  0 &0\\
	z &  x^2  &  r&0\\
	0 &  0  &  y & 0\\
	t&0 & w &y^2\\
	\end{pmatrix} \, \, \, \text{and}\, \, \, \phi_2=\begin{pmatrix} 0&0 &a & 0\\e&0& f & a^{2}\\b&0 &0& 0 \\
   c &b^{2}& d&0\end{pmatrix}.$\end{center}
		Since
	$$
	\phi_i^T\begin{pmatrix}
	0 &  \alpha_1 & \alpha_2&0\\
	-\alpha_1  &  0 & 0 &0\\
	\alpha_3&  0  &0  & \alpha_4\\
	0&0&-\alpha_4&0
	\end{pmatrix} \phi_i=	\begin{pmatrix}
	0 &  \alpha^*_1& \alpha_2^{*}&0\\
	-\alpha^*_1  &  0 & 0 &0\\
	\alpha_3^{*}&  0  &0  & \alpha_4^{*}\\
	0&0&-\alpha_4^{*}&0
	\end{pmatrix} ,
	$$
	 we have that the action of ${\rm Aut} ({\A}_{4,2})$ on the subspace
$\Big\langle \sum\limits_{i=1}^4\alpha_i\nabla_i  \Big\rangle$
is given by
$\Big\langle \sum\limits_{i=1}^4\alpha_i^{*}\nabla_i\Big\rangle,$
where: 
\begin{longtable}{lllll}
if $i=1:$ & $\alpha^*_1=   \alpha_1x^3,$ &
$\alpha^*_2= \alpha_1 x r + \alpha_2xy  - \alpha_4ty  ,$ &
$\alpha^*_3=  -\alpha_1xr + \alpha_3 xy + \alpha_4t y ,$ &
$\alpha_4^*=  \alpha_4y^3;$\\
if $i=2:$ & $\alpha_1^{*}= \alpha_4b^3,$&$
\alpha_2^{*}= 
-\alpha_1 a e + \alpha_3ab + \alpha_4bd,$&
$\alpha_3^{*}=\alpha_1 ae + \alpha_2ab- \alpha_4 bd,$&$
\alpha_4^{*}=\alpha_1a^3$.
\end{longtable}
So, we need to consider only the following cases:
\begin{enumerate}
    \item If $\alpha_2=-\alpha_3$, by choosing $x=\frac{1}{\sqrt[3]{\alpha_1}}$, $r=\frac{\alpha_3}{\alpha_1\sqrt[3]{\alpha_4}}$ and $y=\frac{1}{\sqrt[3]{\alpha_4}}$, we obtain  the representative $\langle \nabla_1+\nabla_4\rangle $.
    
    \item If $\alpha_2\neq -\alpha_3$, by choosing $x=\frac{\alpha_2 + \alpha_3}{\sqrt[3]{\alpha_1^2\alpha_4}}$, $r=\frac{\alpha_3(\alpha_2 + \alpha_3)}{\alpha_1\sqrt[3]{\alpha_1\alpha_4^2}}$ and $y=\frac{\alpha_2 + \alpha_3}{\sqrt[3]{\alpha_1\alpha_4^2}}$, we obtain the representative  $\langle \nabla_1+\nabla_2+\nabla_4\rangle$. 
\end{enumerate}

Hence, we get the distinct orbits  $\langle \nabla_1+\nabla_4\rangle $ and $\langle \nabla_1+\nabla_2+\nabla_4\rangle$, which give us  the following   algebras 
\begin{longtable}{lllllllll}
$\Q_{5,18}$&$:$&$ e_1e_1=e_2$&$e_1e_2=e_5$&$e_2e_1=-e_5$&$ e_3 e_3=e_4$&$ e_3e_4=e_5$&$e_4e_3=-e_5$\\
$\Q_{5,19}$&$:$&$ e_1e_1=e_2$&$e_1e_2=e_5$&$ e_1e_3=  e_5$&$e_2e_1=-e_5$&$ e_3 e_3=e_4$&$ e_3e_4=e_5$&$e_4e_3=-e_5$

\end{longtable}

\subsubsection{Central extensions of ${\A}_{4,7}$}Let us use the following notations:
\begin{longtable}{lll}
$\nabla_1 =[\Delta_{11}],$& 
$\nabla_2 = [\Delta_{12}],$& 
$ \nabla_3 = [\Delta_{14}]-[\Delta_{24}]-[\Delta_{31}].$ 
\end{longtable}
Take $\theta=\sum\limits_{i=1}^3\alpha_i\nabla_i\in {\rm H^2}({\A}_{4,7})$. The automorphism group of ${\A}_{4,7}$ consists of invertible matrices of the form
\begin{center}
 $\begin{pmatrix} x& 0&0 & 0\\0&x& 0 & 0\\y&z & x^{2}& 0 \\
   w & t& 0&x^{2}\end{pmatrix}.$   
\end{center}
		Since
	$$
	\phi^T\begin{pmatrix}
	\alpha_1 &  \alpha_2 & 0&\alpha_3\\
	0  &  0 & 0 &-\alpha_3\\
	-\alpha_3&  0  &0  & 0\\
	0&0&0&0
	\end{pmatrix} \phi=	\begin{pmatrix}
	\alpha_1^{*} &  \alpha_2^{*}-\alpha^{**} & 0&\alpha_3^*\\
	\alpha^*  &  \alpha^{**} & 0 &-\alpha_3^*\\
	-\alpha_3^*&  0  &0  & 0\\
	0&0&0&0
	\end{pmatrix} ,
	$$
	 we have that the action of ${\rm Aut} ({\A}_{4,7})$ on the subspace
$\Big\langle \sum\limits_{i=1}^3\alpha_i\nabla_i  \Big\rangle$
is given by
$\Big\langle \sum\limits_{i=1}^3\alpha_i^{*}\nabla_i\Big\rangle,$
where
\begin{longtable}{llll}
$\alpha^*_1=  \alpha_1 x^2 - \alpha_3yx + \alpha_3wx $, &
$\alpha^*_2= \alpha_2 x^2$, &
$\alpha^*_3=  \alpha_3x^3$.
\end{longtable}

Note that, since $\operatorname{Ann}(\mathcal{A}_{4,7})=\langle e_3,e_4\rangle$,  we can not have $\alpha_3=0$. Suppose $\alpha_3=1$. Then, we there are two cases to consider:
\begin{enumerate}
    \item If $\alpha_2=0$, by choosing  $x=1$ and  $y=\alpha_1$, we get the representative $\langle  \nabla_3\rangle$.
    \item If $\alpha_2\neq 0$, by choosing $x=\alpha_2$ and $y=\alpha_1\alpha_2$, we get the representative $\langle \nabla_2+  \nabla_3\rangle$.
\end{enumerate}

Therefore, we have the distinct orbits $\langle  \nabla_3\rangle$ and $\langle \nabla_2+ \nabla_3\rangle$, which give us the following algebras:  
\begin{longtable}{lllllllllll}
$\Q_{5,20}$ & $:$&$ e_1 e_2 = e_3$&$  e_1e_4= e_5 $&$ e_2 e_1 = e_4$&$  e_2 e_2 =-e_3$&$  e_2 e_4=-e_5$&$  e_3 e_1= -e_5$\\
$\Q_{5,21}$ & $:$&$ e_1 e_2 = e_3+e_5$&$  e_1e_4= e_5 $&$ e_2 e_1 = e_4$&$  e_2 e_2 =-e_3$&$  e_2 e_4=-e_5$&$  e_3 e_1= -e_5$
\end{longtable}
\subsubsection{Central extensions of ${\A}_{4,12}$}Let us use the following notations:
\begin{longtable}{llll}
$\nabla_1 =[\Delta_{11}],$& 
$\nabla_2 = [\Delta_{14}]-[\Delta_{31}],$& 
$ \nabla_3 = [\Delta_{22}],$&
$\nabla_4=[\Delta_{23}]-[\Delta_{42}].$ 
\end{longtable}
Take $\theta=\sum\limits_{i=1}^4\alpha_i\nabla_i\in {\rm H^2}({\A}_{4,12})$. 	The automorphism group of ${\A}_{4,12}$ consists of invertible matrices of the forms
\begin{center}
$\phi_1=\begin{pmatrix} x& 0&0 & 0\\0 &y& 0 & 0\\ r&s & xy& 0\\
  t & z& 0&xy\end{pmatrix}\, \, \, \text{and}\, \, \phi_2= \begin{pmatrix} 0& a&0 & 0\\b&0& 0 & 0\\ c&d & 0& ab\\
 e & f&ab&0\end{pmatrix}. $
\end{center}
		Since
	$$
	\phi_i^T\begin{pmatrix}
	\alpha_1 &  0 & 0&\alpha_2\\
	0  &  \alpha_3 & \alpha_4 &0\\
	-\alpha_2&  0  &0  & 0\\
	0&-\alpha_4&0&0
	\end{pmatrix} \phi_i=	\begin{pmatrix}
	\alpha_1^{*} &  \alpha^* & 0&\alpha_2^{*}\\
	\alpha^{**}&  \alpha_3^{*} & \alpha_4^{*} &0\\
	-\alpha_2^{*}&  0  &0  & 0\\
	0&-\alpha_4^{*}&0&0
	\end{pmatrix}  ,
	$$
	 we have that the action of ${\rm Aut} ({\A}_{4,12})$ on the subspace
$\Big\langle \sum\limits_{i=1}^4\alpha_i\nabla_i  \Big\rangle$
is given by
$\Big\langle \sum\limits_{i=1}^4\alpha_i^{*}\nabla_i\Big\rangle,$
where
\begin{longtable}{lllll}
if $i=1:$ & $\alpha_1^{*}=(\alpha_1x +\alpha_2(t-r))x ,$& $\alpha_2^{*}=\alpha_2x^2y,$& $\alpha_3^{*}= (\alpha_3y + \alpha_4 (s-z))y ,$ & $\alpha_{4}^{*}=\alpha_4xy^2;$\\
if $i=2:$ & $\alpha_1^{*}=(\alpha_3b + \alpha_4 (c - e))b ,$&$
\alpha_2^{*}=\alpha_4ab^2,$&
$\alpha_3^{*}= (\alpha_1 a + \alpha_2 (f-d))a ,$&$ \alpha_{4}^{*}=\alpha_2 a^2 b.$
\end{longtable}
Note that, since $\operatorname{Ann}(\A_{4,12})=\langle e_3, e_4\rangle$, we can not have $\alpha_2=\alpha_4=0$. Besides, considering $\phi_2$ given by $a=b=1$, we see that it change the position of $\alpha_2$ and $\alpha_4$. Thus, we need to consider only the case when $\alpha_2\neq 0$. Therefore, assume that $\alpha_2=1$. 
    \begin{enumerate}
        \item Suppose $\alpha_4=0$.
        
        \begin{enumerate}
        
        \item If $\alpha_3=0$, by choosing $a=b=1$ and $d=\alpha_1$, we have the representative $\langle \nabla_4\rangle$. 
            \item  If $\alpha_3\neq 0$, by choosing $x=y=\alpha_3$ and  $r=\alpha_1\alpha_3$, we obtain the representative $\langle \nabla_2 +\nabla_3\rangle$.
        \end{enumerate}

        \item If $\alpha_4\neq 0$, 
                by choosing $x=\sqrt[3]{\alpha_4}$, $y=\frac{1}{\sqrt[3]{\alpha_4^2}}$, $r=\alpha_1\sqrt[3]{\alpha_4}$ and $s=-\frac{\alpha_3}{\alpha_4\sqrt[3]{\alpha_4^2}}$, we obtain the representative  $ \langle\nabla_2+\nabla_4\rangle$.
    \end{enumerate}

Thus, we have three representatives of distinct orbits
\begin{center}$\langle \nabla_2 +\nabla_3\rangle$,  $ \langle\nabla_2+\nabla_4\rangle$ and  $\langle \nabla_4\rangle$,
\end{center}
which give us the following algebras:
 \begin{longtable}{llllllll}
  $\Q_{5,22}$&$ :$&$   e_1 e_2 = e_3$&$  e_1e_4=e_5$&$  e_2 e_1 = e_4$&$  e_2e_2=e_5$&$ e_3e_1=-e_5$\\
  
 $\Q_{5,23}$&$ :$&$   e_1 e_2 = e_3$&$   e_1e_4=e_5$&$ e_2 e_1 = e_4$&$ e_2e_3=e_5$&$ e_3e_1=-e_5$&$e_4e_2=-e_5$\\

 $\Q_{5,24}$&$: $&$ e_1 e_2 = e_3$&$ e_2 e_1 = e_4$&$ e_2e_3=e_5$&$e_4e_2=-e_5$\\
 
 \end{longtable}
\subsubsection{Central extensions of ${\A}_{4,14}^{\lambda}$}
 Let us use the following notations:
\begin{longtable}{llll}
$\nabla_1 =[\Delta_{11}],$& 
$\nabla_2 = [\Delta_{21}],$& 
$ \nabla_3 = [ \Delta_{13}]-\lambda^2[\Delta_{31}]+\lambda [\Delta_{24}]-[\Delta_{42}],$&$\nabla_4=[\Delta_{23}]-[\Delta_{32}]$. 
\end{longtable}
Take $\theta=\sum\limits_{i=1}^4\alpha_i\nabla_i\in {\rm H^2}({\A}_{4,14}^{\lambda})$. 	The automorphism group of ${\A}_{4,14}^{\lambda}$ consists of invertible matrices of the form
\begin{center}
    $ \phi=\begin{pmatrix} x& a&0 & 0\\0 &y& 0 & 0\\ b&c & y^2& 0\\
 d & e& (1+ \lambda) ay&xy\end{pmatrix}$
\end{center}
		Since
	$$
	\phi^T\begin{pmatrix}
	\alpha_1 &  0 & \alpha_3&0\\
\alpha_2  &  0 & \alpha_4 &\lambda\alpha_3\\
	-\lambda^2\alpha_3&  -\alpha_4 &0  & 0\\
	0&-\alpha_3&0&0
	\end{pmatrix} \phi=\begin{pmatrix}
	\alpha_1^* &  \alpha^* & \alpha_3^*&0\\
\alpha_2^*+\lambda \alpha^*  &  \alpha^{**} & \alpha_4^* &\lambda\alpha_3^*\\
	-\lambda^2\alpha_3^*&  -\alpha_4^* &0  & 0\\
	0&-\alpha_3^*&0&0
	\end{pmatrix},
	$$
	 we have that the action of ${\rm Aut} ({\A}_{4,14}^{\lambda})$ on the subspace
$\Big\langle \sum\limits_{i=1}^4\alpha_i\nabla_i  \Big\rangle$
is given by
$\Big\langle \sum\limits_{i=1}^4\alpha_i^{*}\nabla_i\Big\rangle,$
where
\begin{longtable}{lll}
$   \alpha_1^{*}$&$= $&$ \alpha_1 x^2 + \alpha_3bx(1 - \lambda^2), $\\ 
 $
 \alpha_2^{*}$&$=$&$ \alpha_2x y + \alpha_3ab(1+\lambda^3) + \alpha_4by(1+\lambda)  + \alpha_1ax(1 - \lambda) + 2 \alpha_3dy \lambda  - 
\alpha_3cx\lambda (1 + \lambda),$\\ 
$ \alpha_3^{*}$&$=$&$\alpha_3xy^{2},$\\
$\alpha_4^{*}$&$= $&$ \alpha_4y^3  +\alpha_3 ay^2(1 + \lambda + \lambda^2).$
 \end{longtable}
Since $\operatorname{Ann}(\A_{4,14}^{\lambda})=\la e_3, e_4\ra$, we can not have $\alpha_3=0$. Then,  assume $\alpha_3=1$.
\begin{enumerate}
    \item  $\lambda=0$.
    \begin{enumerate}
        \item Suppose $\alpha_2\neq \alpha_1\alpha_4$.  Then, by choosing $x=y=\alpha_2-\alpha_1\alpha_4$, $a= -\alpha_4(\alpha_2-\alpha_1\alpha_4)$ and $b =-\alpha_1(\alpha_2-\alpha_1\alpha_4) $, we obtain the representative   $\la \nabla_2+\nabla_3\ra$.
        
        \item Suppose $\alpha_2=\alpha_1\alpha_4$. Then, by choosing $x=y=1$, $a=-\alpha_4$ and  $b=-\alpha_1$, we obtain the representative $\la  \nabla_3\ra$.
    \end{enumerate}

    \item $\lambda=1$.
    \begin{enumerate}
    \item \label{caso 1 a415 lambda =1} If $\alpha_1=0$, 
       by choosing  $x=1$, $y=1$, $a=-\frac{\alpha_4}{3}$ and $d=-\frac{\alpha_2}{2}$, we obtain the representative $\la \nabla_3\ra$.

        \item If $\alpha_1 \neq 0$, by choosing $x=\frac{1}{\sqrt{\alpha_1}}$, $y=\sqrt[4]{\alpha_1}$, $a=-\frac{\alpha_4\sqrt[4]{\alpha_1} }{3}$  and  $c=\frac{\alpha_2\sqrt[4]{\alpha_1} }{2}$, we obtain the representative  $ \la \nabla_1+\nabla_3\ra$.

\end{enumerate}

    \item $\lambda=-1$.

  \begin{enumerate}
      \item  If $\alpha_1=0$, by choosing $x=y=1$, $a=-\alpha_4$ and $d=\frac{\alpha_2}{2}$, we obtain the representative $ \la \nabla_3\ra $.
      \item If $\alpha_1 \neq 0$, by choosing  $x=\frac{1}{\sqrt{\alpha_1}}$, $y=\sqrt[4]{\alpha_1}$, $a=-\alpha_4\sqrt[4]{\alpha_1}$ and $d= \frac{\alpha_2-2\alpha_1\alpha_4}{2\sqrt{\alpha_1}}$, we obtain the representative $\la\nabla_1+\nabla_3\ra$.
  \end{enumerate}

    \item $\lambda\notin \{0, \pm 1\}$.
    \begin{enumerate}
     \item If $1+\lambda+\lambda^2\neq0$, by choosing $x=y=1$, $a=-\frac{\alpha_4}{1+\lambda+\lambda^2}$, $b=\frac{\alpha_1}{\lambda^2-1}$, $c=\frac{
 \alpha_2 (\lambda^3-1)+\alpha_4(\alpha_1+2\lambda+\alpha_1\lambda^2)}{\lambda(1+\lambda) (\lambda^3-1)}$ and  $d= \frac{\alpha_4}{\lambda^3-1}$, we obtain the representative $\la\nabla_3\ra$.
    \item Suppose $1+\lambda+\lambda^2=0$.
    \begin{enumerate}
    \item If $\alpha_4=0$, by choosing $x=y=1$, $b=\frac{\alpha_1}{\lambda^2-1}$ and  $c=d=\frac{
 \alpha_2}{\lambda^2-\lambda}$, we obtain the representative $\la \nabla_3\ra$.
        \item If $\alpha_4\neq 0$, by choosing  $x=\alpha_4^2$, $y= \alpha_4$, $b=\frac{\alpha_1\alpha_4^2}{\lambda^2-1}$ and $d=\frac{\alpha_4^2 (\alpha_1\alpha_4  - \alpha_2 ( \lambda-1))}{2\lambda (1-\lambda) }$, we obtain the representative $\la\nabla_3+\nabla_4 \ra$.
    \end{enumerate}

              
    \end{enumerate}

\end{enumerate}

Thus, we get the distinct orbits 
\begin{center}
$\la \nabla_2+\nabla_3\ra_{\lambda=0}$,  $\la \nabla_3\ra$,
$\la \nabla_1+\nabla_3\ra_{\lambda=\pm 1} $   and $\la \nabla_3+\nabla_4\ra_{\lambda =\omega, \omega^2}$,
\end{center}
where $\omega$ is the cubic root of unit, which give us the following algebras:
\begin{longtable}{lllllllll}

$\Q_{5,25}$&$:$&$ e_1 e_2 = e_4$&$e_1e_3=e_5$&$ e_2 e_1 = e_5$&$ e_2 e_2 = e_3$&$ e_4e_2=-e_5$\\
$\Q_{5,26}^{\lambda}$&$:$&$e_1 e_2 = e_4$&$ e_1e_3=e_5$&$ e_2 e_1 = \lambda e_4$&$ e_2 e_2 = e_3$&$ e_2e_4=\lambda e_5 $& $ e_3e_1=-\lambda^2e_5$\\&&$ e_4e_2=-e_5$\\
$\Q_{5,27}$&$:$&$e_1e_1=e_5$&$e_1 e_2 = e_4$&$ e_1e_3=e_5$&$ e_2 e_1 = e_4$&$ e_2 e_2 = e_3$& $ e_2e_4= e_5 $\\
&&$ e_3e_1=-e_5$&$ e_4e_2=-e_5$&\\ 
$\Q_{5,28}$&$:$&$e_1e_1=e_5$&$e_1 e_2 = e_4$&$ e_1e_3=e_5$&$ e_2 e_1 = - e_4$&$ e_2 e_2 = e_3$& $ e_2e_4=- e_5 $\\
&&$ e_3e_1=-e_5$&$ e_4e_2=-e_5$&\\
$\Q_{5,29}$&$:$&$e_1 e_2 = e_4$&$ e_1e_3=e_5$&$ e_2 e_1 = \omega e_4$&$ e_2 e_2 = e_3$&$ e_2e_3=e_5$& $ e_2e_4=\omega e_5 $
\\&&$ e_3e_1=-\omega^2e_5$&$ e_3e_2=-e_5$&$ e_4e_2=-e_5$&\\
$\Q_{5,30}$&$:$&$e_1 e_2 = e_4$&$ e_1e_3=e_5$&$ e_2 e_1 = \omega^2 e_4$&$ e_2 e_2 = e_3$&$ e_2e_3=e_5$& $ e_2e_4=\omega^2 e_5 $\\&&$ e_3e_1=-\omega e_5$&$ e_3e_2=-e_5$&$ e_4e_2=-e_5$&
\end{longtable}


\subsection{Classification theorem}\label{secteoA}
Now we are ready to summarize all results related to the algebraic classification of complex antiassociative algebras.

\begin{theoremA}\label{teorA} Let ${\bf A}$ be a complex  $5$-dimensional non-2-step nilpotent antiassociative algebra. Then ${\bf A}$ is isomorphic to one algebra from the following table:
       \begin{longtable}{lllllll}
           \hline
           $\Q_{5,1}$&$: $&$e_1 e_1=e_2 $&$e_1 e_2 =e_3$&$ e_2 e_1=-e_3$ &&\\
\hline
$\Q_{5,2}$&$:$&$e_1 e_1=e_2 $&$ e_1 e_2 =e_4$&$ e_1 e_3 =  e_4$&$ e_2 e_1=-e_4 $&\\
\hline
$\Q_{5,3}$&$:$&$ e_1 e_1=e_2 $&$ e_1  e_2 =e_4 $&$ e_2  e_1=-e_4 $&$e_3  e_3=e_4$&\\
\hline
$\Q_{5,4}$&$:$&$ e_1  e_1=e_2$&$  e_1  e_2=e_4$&$   e_2  e_1=-e_4 $&$  e_3  e_1 =  e_5$&\\ 
\hline 
$\Q_{5,5}^{\lambda}$&$ :$&$  e_1  e_1=e_2$&$  e_1  e_2=e_4$&$  e_1  e_3 =e_5$&$  e_2  e_1=-e_4 $&$  e_3  e_1 = \lambda e_5$\\
\hline 
$\Q_{5,6}$&$ :$&$ e_1  e_1=e_2$&$  e_1e_2=e_4$&$  e_1e_3=e_5$&$  e_2e_1=-e_4$&$ e_3e_3= e_5$\\
\hline 
$\Q_{5,7}$&$ :$&$ e_1  e_1=e_2$&$  e_1e_2=e_4$&$  e_2e_1=-e_4$&$ e_3e_3= e_5$&\\
\hline 
$\Q_{5,8}$&$ : $&$ e_1  e_1=e_2$&$ e_1e_2=e_4$&$  e_1e_3=e_4+e_5$&$  e_2e_1=-e_4$&$  e_3e_1=-e_5$\\
\hline 
$\Q_{5,9}$&$  :$&$ e_1  e_1=e_2$&$  e_1e_2=e_4$&$e_1e_3=e_4$&$  e_2e_1=-e_4$&$  e_3e_3=e_5$\\
\hline 
$\Q_{5,10}$&$ : $&$ e_1  e_1=e_2$&$ e_1e_2=e_4$&$  e_1e_3=e_4$\\&&$  e_2e_1=-e_4$&$  e_3e_1=e_5$ &$  e_3e_3=e_5$ \\
\hline 
$\Q_{5,11}$&$  :$&$ e_1  e_1=e_2$&$  e_1e_2=e_4$&$  e_2e_1=-e_4$&$  e_3e_1=e_5$&$e_3e_3=e_4$\\
\hline 
$\Q_{5,12}^{\lambda}$&$  :$&$  e_1  e_1=e_2$&$  e_1e_2=e_4$&$  e_1e_3=e_5$\\& &$  e_2e_1=-e_4$&$  e_3e_1=\lambda e_5$ &$ e_3e_3=e_4$ \\
\hline
          $\Q_{5,13}$& $:$ & $e_1e_1=e_2$ & $ e_1e_2=e_5$& $e_1e_4=e_5$ & $e_2e_1=-e_5$ & $ e_3e_3= e_5$ \\
       
\hline
$\Q_{5,14}^{\lambda}$& $:$ &  $e_1e_1=e_2$ &   $ e_1e_2=e_5$& $e_2e_1=-e_5$ & $ e_3e_4=e_5$ & $ e_4e_3=\lambda e_5 $\\
\hline

$\Q_{5,15}$ & $:$ &  $ e_1e_1=e_2$  & $ e_1e_2=e_5$&$  e_1e_3=e_5$\\&& $e_2e_1=-e_5$  & $ e_3e_4=e_5$ &$e_4e_3=-e_5$ \\
\hline
$\Q_{5,16}$& $:$ &  $ e_1e_1=e_2$  & $ e_1e_2=e_5$ & $e_1e_4= e_5$ & $e_2e_1=-e_5$\\&&$e_3e_3=e_5 $  &$ e_3e_4=e_5$ & $e_4e_3=-e_5$ \\
\hline
$\Q_{5,17}$& $:$ &  $ e_1e_1=e_2$  & $ e_1e_2=e_5$& $e_2e_1=-e_5$ \\&& $e_3e_3=e_5 $ &$ e_3e_4=e_5$  & $e_4e_3=-e_5$ \\
\hline
          $\Q_{5,18}$&$:$&$ e_1e_1=e_2$&$e_1e_2=e_5$&$e_2e_1=-e_5$\\&&$ e_3 e_3=e_4$&$ e_3e_4=e_5$ &$e_4e_3=-e_5$ \\
          \hline
$\Q_{5,19}$&$:$&$ e_1e_1=e_2$&$e_1e_2=e_5$&$ e_1e_3=  e_5$&$e_2e_1=-e_5$\\&&$  e_3 e_3=e_4$ &$ e_3e_4=e_5$&$e_4e_3=-e_5$   \\
\hline
         $\Q_{5,20}$ & $:$&$ e_1 e_2 = e_3$&$  e_1e_4= e_5 $&$ e_2 e_1 = e_4$\\&&$  e_2 e_2 =-e_3$&$  e_2 e_4=-e_5$ &$  e_3 e_1= -e_5$ \\
         \hline
$\Q_{5,21}$ & $:$&$ e_1 e_2 = e_3+e_5$&$  e_1e_4= e_5 $&$ e_2 e_1 = e_4$\\&&$  e_2 e_2 =-e_3$&$  e_2 e_4=-e_5$ &$  e_3 e_1= -e_5$ \\
\hline
             $\Q_{5,22}$&$:$&$  e_1 e_2 = e_3$&$  e_1e_4=e_5$&$  e_2 e_1 = e_4$&$ e_2e_2=e_5$&$e_3e_1=-e_5$\\
             \hline
 $\Q_{5,23}$&$:$&$  e_1 e_2 = e_3$&$ e_1e_4=e_5$&$ e_2 e_1 = e_4$\\&&$ e_2e_3=e_5$&$e_3e_1=-e_5$ &$e_4e_2=-e_5$ \\
 \hline
 $\Q_{5,24}$&$: $&$ e_1 e_2 = e_3$&$ e_2 e_1 = e_4$&$  e_2e_3=e_5$&$e_4e_2=-e_5$&\\
 \hline
$\Q_{5,25}$&$:$&$ e_1 e_2 = e_4$&$e_1e_3=e_5$&$ e_2 e_1 = e_5$&$ e_2 e_2 = e_3$&$ e_4e_2=-e_5$\\
\hline 
$\Q_{5,26}^{\lambda}$&$:$&$e_1 e_2 = e_4$&$ e_1e_3=e_5$&$ e_2 e_1 = \lambda e_4$&$ e_2 e_2 = e_3$\\&&$ e_2e_4=\lambda e_5 $ & $ e_3e_1=-\lambda^2e_5$&$ e_4e_2=-e_5$ \\
\hline
$\Q_{5,27}$&$:$&$e_1e_1=e_5$&$e_1 e_2 = e_4$&$ e_1e_3=e_5$&$ e_2 e_1 = e_4$\\&&$ e_2 e_2 = e_3$ & $ e_2e_4= e_5 $&$ e_3e_1=-e_5$& $ e_4e_2=-e_5$ \\
\hline
$\Q_{5,28}$&$:$&$e_1e_1=e_5$&$e_1 e_2 = e_4$&$ e_1e_3=e_5$&$ e_2 e_1 = - e_4$\\&&$ e_2 e_2 = e_3$ & $ e_2e_4=- e_5 $&$ e_3e_1=-e_5$& $ e_4e_2=-e_5$ \\
\hline
$\Q_{5,29}$&$:$&$e_1 e_2 = e_4$&$ e_1e_3=e_5$&$ e_2 e_1 = \omega e_4$&$ e_2 e_2 = e_3$&$ e_2e_3=e_5$\\&&
 $ e_2e_4=\omega e_5 $&$ e_3e_1=-\omega^2e_5$&$ e_3e_2=-e_5$&$ e_4e_2=-e_5$&\\
\hline
$\Q_{5,30}$&$:$&$e_1 e_2 = e_4$&$ e_1e_3=e_5$&$ e_2 e_1 = \omega^2 e_4$&$ e_2 e_2 = e_3$&$ e_2e_3=e_5$\\&& $ e_2e_4=\omega^2 e_5 $&$ e_3e_1=-\omega e_5$&$ e_3e_2=-e_5$&$ e_4e_2=-e_5$&\\
\hline
\hline
\end{longtable}
In the table above, $\omega$ is the primitive cubic root of unit.
\end{theoremA}

\section{The geometric classification of nilpotent  algebras}

\subsection{Definitions and notation}
Given an $n$-dimensional vector space $\mathbb V$, the set ${\rm Hom}(\mathbb V \otimes \mathbb V,\mathbb V) \cong \mathbb V^* \otimes \mathbb V^* \otimes \mathbb V$
is a vector space of dimension $n^3$. This space has the structure of the affine variety $\mathbb{C}^{n^3}.$ Indeed, let us fix a basis $e_1,\dots,e_n$ of $\mathbb V$. Then any $\mu\in {\rm Hom}(\mathbb V \otimes \mathbb V,\mathbb V)$ is determined by $n^3$ structure constants $c_{ij}^k\in\mathbb{C}$ such that
$\mu(e_i\otimes e_j)=\sum\limits_{k=1}^nc_{ij}^ke_k$. A subset of ${\rm Hom}(\mathbb V \otimes \mathbb V,\mathbb V)$ is {\it Zariski-closed} if it can be defined by a set of polynomial equations in the variables $c_{ij}^k$ ($1\le i,j,k\le n$).

Let $T$ be a set of polynomial identities.
The set of algebra structures on $\mathbb V$ satisfying polynomial identities from $T$ forms a Zariski-closed subset of the variety ${\rm Hom}(\mathbb V \otimes \mathbb V,\mathbb V)$. We denote this subset by $\mathbb{L}(T)$.
The general linear group ${\rm GL}(\mathbb V)$ acts on $\mathbb{L}(T)$ by conjugations:
$$ (g * \mu )(x\otimes y) = g\mu(g^{-1}x\otimes g^{-1}y)$$
for $x,y\in \mathbb V$, $\mu\in \mathbb{L}(T)\subset {\rm Hom}(\mathbb V \otimes\mathbb V, \mathbb V)$ and $g\in {\rm GL}(\mathbb V)$.
Thus, $\mathbb{L}(T)$ is decomposed into ${\rm GL}(\mathbb V)$-orbits that correspond to the isomorphism classes of algebras.
Let $\mathcal{O}(\mu)$ denote the orbit of $\mu\in\mathbb{L}(T)$ under the action of ${\rm GL}(\mathbb V)$ and $\overline{\mathcal{O}(\mu)}$ denote the Zariski closure of $\mathcal{O}(\mu)$.

Let $\mathcal A$ and $\mathcal B$ be two $n$-dimensional algebras satisfying the identities from $T$, and let $\mu,\lambda \in \mathbb{L}(T)$ represent $\mathcal A$ and $\mathcal B$, respectively.
We say that $\mathcal A$ degenerates to $\mathcal B$ and write $\mathcal A\to \mathcal B$ if $\lambda\in\overline{\mathcal{O}(\mu)}$.
Note that in this case we have $\overline{\mathcal{O}(\lambda)}\subset\overline{\mathcal{O}(\mu)}$. Hence, the definition of a degeneration does not depend on the choice of $\mu$ and $\lambda$. If $\mathcal A\not\cong \mathcal B$, then the assertion $\mathcal A\to \mathcal B$ is called a {\it proper degeneration}. We write $\mathcal A\not\to \mathcal B$ if $\lambda\not\in\overline{\mathcal{O}(\mu)}$.

Let $\mathcal A$ be represented by $\mu\in\mathbb{L}(T)$. Then  $\mathcal A$ is  {\it rigid} in $\mathbb{L}(T)$ if $O(\mu)$ is an open subset of $\mathbb{L}(T)$.
 Recall that a subset of a variety is called irreducible if it cannot be represented as a union of two non-trivial closed subsets.
 A maximal irreducible closed subset of a variety is called an {\it irreducible component}.
It is well known that any affine variety can be represented as a finite union of its irreducible components in a unique way.
The algebra $\mathcal A$ is rigid in $\mathbb{L}(T)$ if and only if $\overline{O(\mu)}$ is an irreducible component of $\mathbb{L}(T)$.

Given the spaces $U$ and $W$, we write simply $U>W$ instead of $\dim \,U>\dim \,W$.



\subsection{Method of the description of  degenerations of algebras}

In the present work we use the methods applied to Lie algebras in \cite{BC99,GRH,GRH2}.
First of all, if $\mathcal A\to \mathcal B$ and $\mathcal A\not\cong \mathcal B$, then $\mathfrak{Der}(\mathcal A)<\mathfrak{Der}(\mathcal B)$, where $\mathfrak{Der}(\mathcal A)$ is the Lie algebra of derivations of $\mathcal A$. We compute the dimensions of algebras of derivations and check the assertion $\mathcal A\to \mathcal B$ only for such $\mathcal A$ and $\mathcal B$ that $\mathfrak{Der}(\mathcal A)<\mathfrak{Der}(\mathcal B)$.

To prove degenerations, we construct families of matrices parametrized by $t$. Namely, let $\mathcal A$ and $\mathcal B$ be two algebras represented by the structures $\mu$ and $\lambda$ from $\mathbb{L}(T)$ respectively. Let $e_1,\dots, e_n$ be a basis of $\mathbb  V$ and $c_{ij}^k$ ($1\le i,j,k\le n$) be the structure constants of $\lambda$ in this basis. If there exist $a_i^j(t)\in\mathbb{C}$ ($1\le i,j\le n$, $t\in\mathbb{C}^*$) such that $E_i^t=\sum\limits_{j=1}^na_i^j(t)e_j$ ($1\le i\le n$) form a basis of $\mathbb V$ for any $t\in\mathbb{C}^*$, and the structure constants of $\mu$ in the basis $E_1^t,\dots, E_n^t$ are such rational functions $c_{ij}^k(t)\in\mathbb{C}[t]$ that $c_{ij}^k(0)=c_{ij}^k$, then $\mathcal A\to \mathcal B$.
In this case  $E_1^t,\dots, E_n^t$ is called a {\it parametrized basis} for $\mathcal A\to \mathcal B$.
To simplify our equations, we will use the notation $A_i=\langle e_i,\dots,e_n\rangle,\ i=1,\ldots,n$ and write simply $A_pA_q\subset A_r$ instead of $c_{ij}^k=0$ ($i\geq p$, $j\geq q$, $k< r$).

Since the variety of $5$-dimensional antiassociative algebras  contains infinitely many non-isomorphic algebras, we have to do some additional work.
Let $\mathcal A(*):=\{\mathcal A(\alpha)\}_{\alpha\in I}$ be a series of algebras, and let $\mathcal B$ be another algebra. Suppose that for $\alpha\in I$, $\mathcal A(\alpha)$ is represented by the structure $\mu(\alpha)\in\mathbb{L}(T)$ and $B\in\mathbb{L}(T)$ is represented by the structure $\lambda$. Then we say that $\mathcal A(*)\to \mathcal B$ if $\lambda\in\overline{\{\mathcal{O}(\mu(\alpha))\}_{\alpha\in I}}$, and $\mathcal A(*)\not\to \mathcal B$ if $\lambda\not\in\overline{\{\mathcal{O}(\mu(\alpha))\}_{\alpha\in I}}$.

Let $\mathcal A(*)$, $\mathcal B$, $\mu(\alpha)$ ($\alpha\in I$) and $\lambda$ be as above. To prove $\mathcal A(*)\to \mathcal B$ it is enough to construct a family of pairs $(f(t), g(t))$ parametrized by $t\in\mathbb{C}^*$, where $f(t)\in I$ and $g(t)\in {\rm GL}(\mathbb V)$. Namely, let $e_1,\dots, e_n$ be a basis of $\mathbb V$ and $c_{ij}^k$ ($1\le i,j,k\le n$) be the structure constants of $\lambda$ in this basis. If we construct $a_i^j:\mathbb{C}^*\to \mathbb{C}$ ($1\le i,j\le n$) and $f: \mathbb{C}^* \to I$ such that $E_i^t=\sum\limits_{j=1}^na_i^j(t)e_j$ ($1\le i\le n$) form a basis of $\mathbb V$ for any  $t\in\mathbb{C}^*$, and the structure constants of $\mu_{f(t)}$ in the basis $E_1^t,\dots, E_n^t$ are such rational functions $c_{ij}^k(t)\in\mathbb{C}[t]$ that $c_{ij}^k(0)=c_{ij}^k$, then $\mathcal A(*)\to \mathcal B$. In this case  $E_1^t,\dots, E_n^t$ and $f(t)$ are called a parametrized basis and a {\it parametrized index} for $\mathcal A(*)\to \mathcal B$, respectively.

One can also use the following  Lemma, whose proof is the same as the proof of Lemma 1.5 from \cite{GRH}.

\begin{lemma}\label{gmain}
Let $\mathfrak{B}$ be a Borel subgroup of ${\rm GL}(\mathbb V)$ and $\mathcal{R}\subset \mathbb{L}(T)$ be a $\mathfrak{B}$-stable closed subset.
If $\mathcal A(*) \to \mathcal B$ and for any $\alpha\in I$ the algebra $\mathcal A(\alpha)$ can be represented by a structure $\mu(\alpha)\in\mathcal{R}$, then there is $\lambda\in \mathcal{R}$ representing $\mathcal B$.
\end{lemma}

\subsection{The geometric classification of  
antiassociative algebras}The main result of the present section is the following theorem.
\begin{theoremB}
The variety of $4$-dimensional antiassociative algebras has dimension $12$ and it has $3$ irreducible components defined by  
\begin{center}$
\mathcal{C}_1=\overline{\{\mathcal{O}(\A_{4,8}^{\gamma})\}}, \, 
\mathcal{C}_2= \overline{\{\mathcal{O}(\A_{4,9}^{\alpha})\}}, \,   
\mathcal{C}_3=\overline{\mathcal{O}(\Q_{4,3})}.$
\end{center}
In particular, there is only one rigid $4$-dimensional antiassociative algebra $\Q_{4,3}$.

The variety of $5$-dimensional antiassociative algebras has dimension $24$ and it has $8$ irreducible components defined by 
\begin{center}
$\mathcal{C}_1=\overline{\{\mathcal{O}({\mathfrak V}_{3+2})\}},$ 
 $\mathcal{C}_2=\overline{\{\mathcal{O}({\mathfrak V}_{4+1})\}},$
 $\mathcal{C}_3=\overline{\mathcal{O}(\Q_{5,10})}, $ 
 $\mathcal{C}_4=\overline{\{\mathcal{O}(\Q_{5,14}^{\lambda})\}}, $\\
 $\mathcal{C}_5=\overline{\mathcal{O}(\Q_{5,19})}, $ 
 $\mathcal{C}_6=\overline{\mathcal{O}(\Q_{5,21})}, $ 
 $\mathcal{C}_7=\overline{\mathcal{O}(\Q_{5,23})}, $
 $\mathcal{C}_{8}=\overline{\{O(\Q_{5,26}^{\lambda})\}},$ 
 \end{center}
 In particular, there are only   four rigid algebras in this variety:
 \begin{center}
     $\Q_{5,10},$ 
 $\Q_{5,19},$ 
 $\Q_{5,21}, $  
and $\Q_{5,23}.$ 
 
 \end{center}
\end{theoremB}
\begin{Proof}
Thanks to \cite{kppv}, the variety of $4$-dimensional $2$-step nilpotent antiassociative algebras has dimension $12$ and it has $2$ irreducible components. All needed degenerations are given in the table below:

\begin{longtable}{|lcllllll|}
\hline 
$\Q_{4,2}$&$\to$&$\Q_{4,1}$ & 
$E_1^t=e_1$ & $E_2^t=e_2$&$E_3^t=e_4$ & $ E_4^t= t e_3$& \\
\hline 
$\Q_{4,3}$&$\to$&$\Q_{4,2}$ & 
$E_1^t=e_1+ \frac{1}{2t}e_2$ & $E_2^t=e_3 +\frac{1}{4t^2}e_4 $&$E_3^t=t e_2+\frac{1}{2}e_3$ & $ E_4^t= e_4$ &\\ \hline

\end{longtable}
Since  $\overline{\mathcal{O}(\Q_{4,3})}$ has dimension $12$,  the variety of $4$-dimensional antiassociative algebras has $3$ irreducible components.

\medskip

Thanks to \cite[Theorem A]{ikp20}, the variety of $5$-dimensional $2$-step nilpotent antiassociative algebras has dimension $24$ and it has $3$ irreducible components.

\begin{longtable}{lllllll}
${\mathfrak V}_{4+1}$ & $:$&  
$e_1e_2=e_5$& $e_2e_1=\lambda e_5$ &$e_3e_4=e_5$&$e_4e_3=\mu e_5$\\

${\mathfrak V}_{3+2}$ &$ :$&
$e_1e_1 =  e_4$& $e_1e_2 = \mu_1 e_5$ & $e_1e_3 =\mu_2 e_5$& 
$e_2e_1 = \mu_3 e_5$  & $e_2e_2 = \mu_4 e_5$  \\
& & $e_2e_3 = \mu_5 e_5$  & $e_3e_1 = \mu_6 e_5$  & \multicolumn{2}{l}{$e_3e_2 = \lambda e_4+ \mu_7 e_5$ } & $e_3e_3 =  e_5$  \\

${\mathfrak V}_{2+3}$ &$ :$&
$e_1e_1 = e_3 + \lambda e_5$& $e_1e_2 = e_3$ & $e_2e_1 = e_4$& $e_2e_2 = e_5$

\end{longtable}

Thanks to the first part of this theorem, 
all $5$-dimensional split antiassociative algebras are in orbit closure 
of families ${\mathfrak V}_{4+1}$ and ${\mathfrak V}_{3+2},$
and 
$\overline{\mathcal{O}(\Q_{5,3})}.$

Let us give some important degenerations.
 
\begin{longtable}{|lcll|}

\hline

$\Q_{5,10}$&$\to$& $\Q_{5,5}^\lambda$ &$E_{1}^{t}=\frac{t}{\lambda -1}e_1+\frac{t}{(\lambda -1)^2}e_3$ \\
\multicolumn{3}{|l}{$E_{2}^{t}=\frac{t^2}{(\lambda -1)^2}e_2+\frac{t^2}{(\lambda -1)^3}e_4+\frac{\lambda t^2}{(\lambda -1)^4}e_5$} & $E_{3}^{t}= \frac{t^3}{(\lambda -1)^2}e_3$ \\
\multicolumn{3}{|l}{$E_{4}^{t}=\frac{t^3}{(\lambda -1)^3}e_4$} &$E_{5}^{t}=\frac{t^4}{(\lambda -1)^4}e_5$ \\
\hline
$\Q_{5,24}$&$\to$& $\Q_{5,8}$ &$E_{1}^{t}=e_1+e_2$ \\
\multicolumn{3}{|l}{$E_{2}^{t}= e_3+e_4$}  &$E_{3}^{t}=te_2+(t-1)e_4$ \\
\multicolumn{3}{|l}{$E_{4}^{t}=e_5$}  &$E_{5}^{t}=-te_4-e_5$ \\
\hline  
$\Q_{5,10}$&$\to$& $\Q_{5,12}^\lambda$ &$E_{1}^{t}=\frac{\lambda^2t^2}{1-\lambda^2}e_1-\frac{\lambda^2t^2}{(\lambda-1)^2(1+\lambda)}e_3$ \\
\multicolumn{3}{|l}{$E_{2}^{t}=\frac{\lambda^3t^4}{(\lambda-1)^2(1+\lambda)}e_3+\frac{\lambda^4t^4}{(\lambda-1)^3(1+\lambda)^2}e_4+\frac{\lambda^5t^4}{(\lambda-1)^4(1+\lambda)^2}e_5 $ } & $E_{3}^{t}=\frac{\lambda^4t^3}{(\lambda^2-1)^2}e_2-\frac{\lambda^3t^3}{(\lambda-1)^2(1+\lambda)}e_3$\\
\multicolumn{3}{|l}{$E_{4}^{t}=\frac{\lambda^6t^6}{(\lambda-1)^4(1+\lambda)^2}e_5$ }  & $E_{5}^{t}=\frac{\lambda^5t^5}{(\lambda^2-1)^3}e_4+\frac{\lambda^5t^5}{(\lambda-1)^4(1+\lambda)^2}e_5$\\
\hline
$\Q_{5,14}^{t-1}$&$\to$& $\Q_{5,15}$ &$E_{1}^{t}=te_1+t^2e_3$\\
\multicolumn{3}{|l}{ $E_{2}^{t}=t^2e_2 $} & $E_{3}^{t}=t(t-1)e_2+e_4$ \\
\multicolumn{3}{|l}{$E_{4}^{t}=-t^3e_3$ }  & $E_{5}^{t}=t^3e_5$ \\
\hline 
$\Q_{5,26}^{1+t}$&$\to$& $\Q_{5,16}$ &$E_{1}^{t}=e_1+e_2$\\
\multicolumn{4}{|l|}{$E_{2}^{t}=\frac{2t^2+3t-3}{2t-1}e_4+\frac{(3 + 3 t + t^2) (2t^2+3t-3)^2}{18 (1 - 2 t)^2 t^2 (1 + t) (2 + t)}e_5$}  \\
\multicolumn{4}{|l|}{$E_{3}^{t}=3t e_2+\frac{2t^2+3t-3}{3t^2(2t^2+3t-2)}e_3-\frac{2 t^4+5 t^3 +2 t^2  -3}{6t^2(2t^2+t-1)}e_4$} \\
\multicolumn{3}{|l}{$E_{4}^{t}=\frac{2t^2+3t-3}{3t}e_3+\frac{3 - 3 t - 2 t^2}{3 t - 6 t^2}e_4-\frac{(3 + 3 t + t^2) (2t^2+3t-3)^3}{54 (1 - 2 t)^2 t^3 (1 + t) (2 + t)}e_5$}&  $E_{5}^{t}=\frac{2t^2+3t-3}{2t-1}e_5$ \\

\hline
$\Q_{5,19}$&$\to$& $\Q_{5,17}$ &$E_{1}^{t}=\frac{1}{2\sqrt[6]{t^5}}e_1-\frac{\sqrt[3]{t}}{2}e_3$\\
\multicolumn{3}{|l}{$E_{2}^{t}=\frac{1}{4\sqrt[3]{t^5}}e_2-\frac{1}{2\sqrt[12]{t^5}}e_3-\frac{1}{4\sqrt[12]{t^{19}}}e_4$} & $E_{3}^{t}= \frac{1}{2\sqrt[3]{t^2}}e_3+\frac{1}{4\sqrt[6]{t^{11}}}e_4$ \\
\multicolumn{3}{|l}{$E_{4}^{t}=\frac{1}{4\sqrt[6]{t^{11}}}e_4-\frac{1}{8t^3}e_5$} & $E_{5}^{t}=\frac{1}{8\sqrt{t^{5}}}e_5$ \\
\hline
 
$\Q_{5,19}$&$\to$& $\Q_{5,18}$ &$E_{1}^{t}=\frac{1}{2t^2}e_1+\frac{1}{2t}e_3$\\
\multicolumn{3}{|l}{$E_{2}^{t}=\frac{1}{4t^4}e_2+\frac{1}{2t}e_3-\frac{1}{4t^2}e_4$} & $E_{3}^{t}= \frac{1}{2t^2}e_3-\frac{1}{2t^3}e_4$ \\
\multicolumn{3}{|l}{$E_{4}^{t}=\frac{1}{4t^4}e_4+\frac{1}{8t^5}e_5$} &$E_{5}^{t}=\frac{1}{8t^6}e_5$ \\
\hline
$\Q_{5,21}$&$\to$& $\Q_{5,25}$ &$E_{1}^{t}=t^2 (t^2-1)e_2+t^2 (t^2-1)e_4$\\
\multicolumn{3}{|l}{$E_{2}^{t}=(t^5-t^3+t^2-1)e_1+(t^2-1)e_2$ } & $E_{3}^{t}= (1-t^2+t^3-t^5)e_4+(1-t^2)e_5$\\
\multicolumn{4}{|l|}{$E_{4}^{t}=-t^2 (t^2-1)^2e_3+(t-t^3+t^4-t^6)e_4+(t-t^2-t^3+2 t^4-t^6)e_5$ } \\  \multicolumn{4}{|l|}{$E_{5}^{t}=(t-1)^2(1+t)^3 (t^2-t^3+t^4)e_5$}\\
\hline
$\Q_{5,26}^{\frac{1}{1+t}}$&$\to$& $\Q_{5,27}$ &$E_{1}^{t}=(1+t)e_1-\frac{t}{2}e_2+\frac{1+t}{2}e_3$\\
\multicolumn{3}{|l}{ $E_{2}^{t}=\sqrt{t}e_2$} & $E_{3}^{t}= te_3+t(1+t)e_4$   \\
\multicolumn{3}{|l}{$E_{4}^{t}=\frac{\sqrt{t}(1+t)(2+t)}{2}e_4$ } & $E_{5}^{t}=\frac{t(2+t)}{2}e_5$ \\
\hline

$\Q_{5,26}^{\frac{\sqrt{-t (8+7 t)}-4 -5t}{\sqrt{-t (8+7 t)}+4+3 t}}$&$\to$& $\Q_{5,28}$ & \\ \multicolumn{4}{|l|}{$E_{1}^{t}=e_1+\frac{\left(\sqrt{-t (8+7 t)}-t\right) \left(4+3 t+\sqrt{-t (8+7 t)}\right)}{16 \sqrt[4]{t} (1+t)}e_2+\frac{\left(4+3 t+\sqrt{-t (8+7 t)}\right)^3}{64(1+t)(\sqrt{-t (8+7 t)}-t)}e_3$}\\
\multicolumn{3}{|l}{$E_{2}^{t}=\frac{4+3 t+\sqrt{-t (8+7 t)}}{4}e_2$ }&  
$E_{3}^{t}= \frac{\left(4+3 t+\sqrt{-t (8+7 t)}\right)^2}{16}e_3+\sqrt[4]{t}(1+t)e_4 $ \\
\multicolumn{3}{|l}{$E_{4}^{t}=(1+t)e_4$} & 
$E_{5}^{t}=\frac{4+3 t+\sqrt{-t (8+7 t)}}{4}e_5$    \\
\hline
$\Q_{5,26}^{\frac{t+h}{1+t}}$&$\to$& $\Q_{5,29},\, \Q_{5,30}$ & $E_{1}^{t}=-\frac{3 i t(1 + h)  }{\sqrt{1+t}}e_2$\\ 
\multicolumn{3}{|l}{$E_{2}^{t}=-\frac{i (1 + t (5 + 2 h))}{3 t (1 + h)\sqrt{ (1 + t) }}e_1+i\sqrt{1 + t}e_2$ }&  
$E_{3}^{t}=-e_3+\frac{1+3t}{3t}e_4 $ \\
\multicolumn{3}{|l}{$E_{4}^{t}=e_4$} & 
$E_{5}^{t}=\frac{i }{\sqrt{1 + t }}e_5$    \\

\multicolumn{4}{|l|}{
$h$ is $\omega$  (for $\Q_{5,29}$)  or $\omega^2$ (for $\Q_{5,30}$), where $\omega$ is the primitive cubic root of unit}\\

\hline
$\Q_{5,26}^{\frac{(1 + L) (t^2-1)}{L-1 + t^2(1 + L)}}$&$\to$&${\mathfrak V}_{2+3}$ & \\
\multicolumn{4}{|l|}{$E_{1}^{t}=\frac{1+L}{2L}e_1+\frac{(t-1)^2 (1 + t) + 2 L (t^3-1) + L^2 (1 + t + t^2 + t^3)}{2 L (1 + L)}e_2-\frac{(L-1 + t^2(1 + L))^3}{8 L^2 t (t^2(1 + L)-1)}e_3$}\\
\multicolumn{4}{|l|}{$E_{2}^{t}=\frac{1}{L}e_1+\frac{(t-1)(L-1 + t^2(1+L))}{L(1 + L)}e_2-\frac{(L-1 + t^2(1+L))^3}{4 t L^2(1+L)((1 + L) t^2-1)}e_3$ } \\
\multicolumn{4}{|l|}{$E_{3}^{t}= \frac{(1-L)(L-1 + t^2(1+L))^2}{2(L+L^2)^2}e_3+ \frac{(1-L)(2+t)}{2L^2}e_4 - \frac{L-1 + t^2(1+L)}{2L^2t}e_5$ } \\
\multicolumn{4}{|l|}{$E_{4}^{t}=\frac{(1-L)(L-1 + t^2(1+L))^2}{2(L+L^2)^2}e_3+ \frac{2+ t + L^2 (2 - t) }{2L^2(1+L)}e_4 - \frac{L-1 + t^2(1+L)}{2L^2t}e_5$}\\
\multicolumn{3}{|l}{$E_{5}^{t}=\frac{(L-1 + t^2(1+L))^2}{(L+L^2)^2}e_3+ \frac{2 + t + L t}{L^2 + L^3}e_4 - \frac{L-1 + t^2(1+L)}{tL^2(1 + L)}e_5;$  }& $\qquad \lambda=\frac{L^2-1}{4}$ \\
\hline
\end{longtable}

For the rest of degenerations, in  case of  $E_1^t,\dots, E_n^t$ is a {\it parametric basis} for ${\bf A}\to {\bf B},$ it will be denote as
${\bf A}\xrightarrow{(E_1^t,\dots, E_n^t)} {\bf B}$.

\begin{longtable}{|lcl|lcl|}
  
\hline
$\Q_{5,2}$ & $  \xrightarrow{ (e_1,e_2,e_4,te_3,e_5)}$ & $\Q_{5,1}$ & 
$\Q_{5,15}$&$ \xrightarrow{ (e_1+e_3,e_2+e_5,e_3,e_5,te_4)}$& $\Q_{5,2}$ \\
\hline

$\Q_{5,11}$& $ \xrightarrow{(e_1,e_2+te_3, e_3,e_4+e_5,\frac{1}{t}e_5)}$ & $\Q_{5,3}$ &

$\Q_{5,10}$&$ \xrightarrow{ (e_1,e_2,t e_3,e_4,te_5)}$& $\Q_{5,4}$ \\
\hline

$\Q_{5,10}$& $\xrightarrow{(-\frac{1}{t}e_1+\frac{1}{t}e_3,\frac{1}{t^2}e_2,\frac{1}{t}e_3,-\frac{1}{t^3}e_4,\frac{1}{t^2}e_5)}$& $\Q_{5,6}$ &
 $\Q_{5,9}$&$\xrightarrow{(e_1,e_2,t e_3,e_4,t^2 e_5)}$& $\Q_{5,7}$ \\
\hline

 $\Q_{5,10}$&$\xrightarrow{(\frac{1}{t}e_1,\frac{1}{t^2}e_2,\frac{1}{t^2}e_3,\frac{1}{t^3}e_4,\frac{1}{t^4}e_5)}$& $\Q_{5,9}$ &

$\Q_{5,10}$&$\xrightarrow{ (-t^2e_1,t^4 e_3,t^3e_2-t^3e_3,t^6e_5,t^5e_4+t^5e_5)}$& $\Q_{5,11}$\\
\hline
$\Q_{5,16}$&$\xrightarrow{(t^2e_1,t^4e_2,t^3 e_3,t^4e_4,t^6 e_5)}$& $\Q_{5,13}$&
$\Q_{5,21}$&$\xrightarrow{(\frac{1+t}{t} e_1+e_2,\frac{1}{t}e_2,\frac{1}{t^2} e_3,\frac{1+t}{t^2}e_4,\frac{1+t}{t^3}e_5)}$& $\Q_{5,20}$ \\
\hline

$\Q_{5,21}$&$\xrightarrow{(te_1,t^2e_2,t^3 e_3+t^3e_5,t^3e_4,t^4e_5)}$& $\Q_{5,22}$&

$\Q_{5,23}$&$\xrightarrow{(t e_1,e_2,t e_3,te_4,t e_5)}$& $\Q_{5,24}$\\
\hline
\end{longtable}

  After a careful  checking  dimensions of orbit closures of the more important algebras for us, we have 

\begin{center}  
$\dim  \mathcal{O}({\mathfrak V}_{3+2})=24,$ \, 
$
\dim \mathcal{O}({\mathfrak V}_{4+1})=20,$\, 
$\dim \mathcal{O}(\Q_{5,10})=20,$\, 
$\dim \mathcal{O}(\Q_{5,14}^{\lambda})=20,$\, 
$\dim \mathcal{O}(\Q_{5,19})=20,$\, 
$\dim \mathcal{O}(\Q_{5,21})=20,$\, 
$\dim \mathcal{O}(\Q_{5,23})=20,$\, 
$\dim \mathcal{O}(\Q_{5,26}^{\lambda})=20.$ \\
\,

 \end{center}
 Hence, 
${\mathfrak V}_{4+1}, \Q_{5,10},  \Q_{5,14}^{\lambda}, \Q_{5,19}, \Q_{5,21}, \Q_{5,23}, \Q_{5,26}^{\lambda}$ and ${\mathfrak V}_{3+2}$ give $8$ irreducible components.

\end{Proof}

\end{document}